\documentclass[11pt,reqno]{amsart}
\usepackage[tmargin=1in,bmargin=1in,rmargin=1in,lmargin=1in]{geometry}
\usepackage{amsmath,amsthm,amsfonts,mathrsfs,latexsym,amssymb,pstricks,pst-grad, graphicx,tabularx, setspace,pstricks, graphics,enumerate,fancyhdr}
\usepackage[breaklinks=true]{hyperref}
\usepackage{xcolor}
\usepackage{tikz}
\DeclareMathOperator{\IN}{i}
\DeclareMathOperator{\UIN}{ui}
\DeclareMathOperator{\di}{d}
\DeclareMathOperator{\diam}{diam}
\DeclareMathOperator{\gr}{gr}
\DeclareMathOperator{\gc}{gcd}
\DeclareMathOperator{\dic}{Dic}
\DeclareMathOperator{\dih}{D}
\DeclareMathOperator{\qdic}{QD}
\newtheorem{theorem}{Theorem}[section]

\newtheorem{corollary}[theorem]{Corollary}
\newtheorem{lemma}[theorem]{Lemma}
\newtheorem{proposition}[theorem]{Proposition}

\newtheorem{definition}[theorem]{Definition}

\newtheorem{utheorem}{\textrm{\textbf{Theorem}}}

\newtheorem{example}[theorem]{Example}

\begin{document}
\title[Uniform Intersection Numbers and Subgroup-Based Generating Graphs of Groups]{Uniform Intersection Numbers and Subgroup-Based Generating Graphs of Groups}

\author[Projesh Nath Choudhury]{Projesh Nath Choudhury}
\author[Biswaranjan Khanra]{Biswaranjan Khanra}
\address{Projesh Nath Choudhury}\address{Department of Mathematics, Indian Institute of Technology, Gandhinagar, Gujarat-382355, India}
\email{projech.nc@iitgn.ac.in}

\address{Biswaranjan Khanra,}
\address{Department of Mathematics, Indian Institute of Technology, Gandhinagar, Gujarat-382355, India}
\email{biswaranjanmath91@gmail.com, kbiswaranjan@iitgn.ac.in}
 	\date{\today}

	\begin{abstract}
		The \emph{uniform intersection number} of a finite group $G$, denoted by $\UIN(G)$, is the smallest positive integer $k$ such that the intersection of any $k$ distinct maximal subgroups of $G$ is equal to its Frattini subgroup. We determine the uniform intersection number of dicyclic, dihedral, and quasi-dihedral groups. We then investigate the uniform intersection number through the subgroup-based generating graph $\mathcal{G}(G)$ of a finite group $G$. We characterize the finite groups $G$ for which both the intersection number and the uniform intersection number are equal to two, and show that this condition is equivalent to $\mathcal{G}(G)$ being complete multipartite. We also classify finite groups in terms of the diameter, girth, and perfectness of $\mathcal{G}(G)$. Finally, we show that $\mathcal{G}(\mathbb{Z}_n)$ is Hamiltonian if and only if $n=p_1^{r}p_2^{r}$, where $r\geq 2$ and $p_1$ and $p_2$ are distinct primes.

	\end{abstract}
	
	.
\maketitle

\section{Introduction and main results} 

\textit{Throughout this paper, for a finitely generated group $G$ and a proper subgroup $S$ of $G$, we define  $\mathcal{M}(G):$ the set of all maximal subgroups of a group $G$,  $M(S)=\{m\in\mathcal{M}(G):S\leq m\}$ the set of all maximal subgroups of $G$ that contains $S$, and $\Phi(G)=\bigcap\limits_{\tiny{m\in\mathcal{M}(G)}}m$ the Frattini subgroup of $G$.}

\textit{All graphs in this paper are assumed to be simple and undirected. For a graph $\mathcal{G}$, $\diam (\mathcal{G})$ and $\gr(\mathcal{G})$ denote the diameter and the girth of $\mathcal{G}$, respectively. For a fixed integer $n \geq 1$, we define
	$[n] := \{ 1, \dots, n \}$.}\bigskip

The intersection number of a finite group $G$, denoted by $\IN(G)$, is the
minimum number of maximal subgroups of $G$ whose intersection is equal to the
Frattini subgroup $\Phi(G)$. The notion of intersection number was introduced
by Archer--Serrano--Cook--Lauderdale--Perez--Villalobo in 2019 \cite{archer}. In
\cite{archer}, the authors determined the intersection number for several
classes of finite groups, including nontrivial nilpotent groups, dihedral
groups, and dicyclic groups. Burness--Garonzi--Lucchini further investigated
the intersection number of finite groups, first in the setting of finite
simple groups in connection with minimal dimension, and subsequently for
almost simple and arbitrary finite groups \cite{burn,burn20}.

A natural question arising from the definition of \(\IN(G)\) is whether the same number of maximal subgroups always suffices, regardless of which maximal subgroups are chosen. The groups $\operatorname{Q}_8$ and $\dic_3$ provide simple examples where this stronger property holds. Indeed, $\IN(\operatorname{Q}_8)=2$ and $\IN(\dic_3)=2$.  Moreover, in both $\operatorname{Q}_8$ and $\operatorname{Dic}_3$, the intersection of any two distinct maximal subgroups is equal to the Frattini subgroup. Thus, in
these examples, the minimum number of maximal subgroups required to obtain
the Frattini subgroup also suffices uniformly for every choice of maximal
subgroups. This uniform behavior does not hold in general. For instance, $\IN(\dic_{30})=4$,
but the intersection of an arbitrary set of four distinct maximal subgroups
of $\dic_{30}$ need not be equal to
$\Phi(\dic_{30})$. On the other hand, the intersection of any
ten distinct maximal subgroups of $\dic_{30}$ is equal to
$\Phi(\dic_{30})$ \cite{hum}. This gap between minimal and uniform intersection behavior motivates us to define the notion of uniform intersection number of a finite group.

\begin{definition}
	Let $G$ be a finite group. The \emph{uniform intersection number} of $G$,
	denoted by $\UIN(G)$, is the smallest positive integer $k$ such that the
	intersection of any $k$ distinct maximal subgroups of $G$ is equal to the
	Frattini subgroup $\Phi(G)$.
\end{definition}

Clearly, $\IN(G)\leq \UIN(G)\leq |\mathcal{M}(G)|$. For instance, $\IN(\operatorname{Dic}_{30}) <\UIN(\operatorname{Dic}_{30})=10.$
One of the main goal of this article is to characterize the finite
groups $G$ for which
\[
\IN(G)=\UIN(G).
\]
To investigate this problem, we use a graph-theoretic approach based on the
subgroup structure of $G$.

Graphs associated with algebraic structures have attracted considerable
attention in algebraic combinatorics over the past two decades
\cite{aba,ali,cameron1,helf}. Such graphs provide a useful framework for
studying structural properties of groups and for exploring the interaction
between algebraic and graph-theoretic invariants. Among the various graphs
associated with groups are Cayley graphs \cite{kelb}, commuting graphs, deep
commuting graphs \cite{cameron}, power graphs \cite{ali}, difference graphs
\cite{biswas}, generating graphs \cite{cameron2,lucchini,lucchini1},
permutability graphs \cite{bolin}, conjugacy class graphs
\cite{bian,bert,chil}, and character degree graphs \cite{isaac}.

The subgroup structure of a group plays a fundamental role in understanding
the group itself \cite{baer,suzuki}. Accordingly, several graphs have been
defined on the system of subgroups of a group, including intersection graphs,
inclusion graphs, comaximal graphs of subgroups \cite{akbari}, and
subgroup-based generating graphs \cite{wei,lu}. In \cite{wei}, the following
graph was introduced for finite groups.

\begin{definition}
	Let $G$ be a finite group. The graph $\Gamma(G)$ has as its vertices the
	nontrivial proper subgroups of $G$, and two distinct vertices $S_1$ and $S_2$
	are adjacent if and only if $ \langle S_1,S_2\rangle=G$.
\end{definition}

In \cite{wei}, the authors characterized the finite groups for which
$\Gamma(G)$ is connected, a star, or a tree, as well as those for which it
has a universal vertex.  Subsequently, Lu--Meng--Qin--wei--Zhang\ \cite{lu} characterized the finite groups for which $\Gamma(G)$ is planar. They further studied the graph
$\Gamma^*(G)$, obtained from $\Gamma(G)$ by deleting its isolated vertices,
and characterized the finite groups for which $\Gamma^*(G)$ is complete, a
star, or has a universal vertex. For a finite group $G$, the isolated vertices of $\Gamma(G)$ are precisely
the nontrivial subgroups contained in the Frattini subgroup $\Phi(G)$. Thus,
the Frattini subgroup provides a natural connection between the intersection
properties of maximal subgroups and the subgroup-based generating graph. 

The main objective of this article is to investigate the uniform intersection number of finite groups through the subgroup-based generating graph. To this end, we study several structural properties of this graph and show that some of our results remain valid for finitely generated groups, thereby broadening the scope of the theory beyond the finite setting. Motivated by the preceding discussion, we introduce the following graph associated with a finitely generated group.

\begin{definition}
	Let $G$ be a finitely generated group. The \emph{subgroup-based generating graph} of $G$, denoted
	by $\mathcal{G}(G)$, is the graph whose vertices are the nontrivial proper
	subgroups of $G$ that are not contained in $\Phi(G)$. Two distinct vertices
	$S_1$ and $S_2$ are adjacent if and only if $ \langle S_1,S_2\rangle=G$.
\end{definition}
For a finite group $G$, the graph $\mathcal{G}(G)$ coincides with
$\Gamma^*(G)$. Thus, $\mathcal{G}(G)$ may be viewed as a natural extension of
$\Gamma^*(G)$ beyond the finite setting.

\medskip

We now state our main results.  In Proposition \ref{propdiam}, we show that for any finitely generated group $G$, the diameter of $\mathcal{G}(G)$ is at most three. In our first main result, we classify the family of finite abelian groups in terms of the diameter of the subgroup based generating graph.
\begin{utheorem}\label{finiteabedia}
	Let $G$ be a finite abelian group. Then the following holds:
	\begin{enumerate}[(1)]
		\item $\diam(\mathcal{G}(G))=1$ if and only if $G\cong \mathbb{Z}_{pq}$ or $G\cong \mathbb{Z}_p\times \mathbb{Z}_p$, where $p$ and $q$ are distinct primes.
		\item  $\diam(\mathcal{G}(G))=2$ if and only if $G\cong \mathbb{Z}_{p^{n_1}}\times \mathbb{Z}_{p^{n_2}}$ with $n_1n_2\geq 2$, or $G\cong \prod_{i=1}^{n_1}\mathbb{Z}_{p^{t_{1i}}}$ with $n_1\geq 2$ and $\sum_{i=1}^{n_1}t_{1i}\geq 3$. 
		
		\item  $\diam(\mathcal{G}(G))=3$ if and only if the order of $G$ is divisible by at least three distinct primes or $G\cong \prod_{i=1}^{n_1}\mathbb{Z}_{p_1^{t_{1i}}}\times \prod_{i=1}^{n_2}\mathbb{Z}_{p_2^{t_{1i}}} $, where $n_1n_2\geq 2$ and $p_1$, $p_2$ are distinct primes.
	\end{enumerate}
\end{utheorem}
In our next main result, we characterize finite groups in terms of the girth of their subgroup-based generating graphs.
\begin{utheorem}\label{corfggirth}
	Let $G$ be a finite group such that $\mathcal{G}(G)$ is nontrivial. Then 
	\begin{equation*}
		\gr (\mathcal{G}(G))=\begin{cases}
			4~ &\text{if}~ G\cong \mathbb{Z}_{p_1^{t_1}p_2^{t_2}} ~~\text{ with } t_1,t_2\geq 2,~\\
			\infty~ &\text{if}~G\cong \mathbb{Z}_{{p_1^{t_1}p_2^{t_2}}}~ \text{ with } t_1=1 \text{ or } t_2=1,~\\ 
			3~ &\text{otherwise},
		\end{cases}
	\end{equation*} 
	where \(p_1\) and \(p_2\) are distinct primes.
\end{utheorem}

In our next main result, we classify the finite groups for which $\IN(G)=\UIN(G)=2$. We show that this class coincides precisely with the family of finite groups whose subgroup-based generating graphs are complete multipartite.
\begin{utheorem}\label{thmcmg}
	Let $G$ be a finite group. Then the following are equivalent.
	\begin{enumerate}[(1)]
		\item  $\IN(G)=\UIN(G)=2$.
		\item $\mathcal{G}(G)$ is a complete multipartite graph.
		\item $G$ is isomorphic to one of the following  groups:
		\begin{enumerate}[(a)]
			\item  $G\cong \mathbb{Z}_{p^s}\times \mathbb{Z}_{p^t}$, where $s+t=k$ such that $\vert G\vert =p^k$ or $G$ is a $2$-generated non abelian $p$-group.
			\item  $G\cong \mathbb{Z}_{p_1^{n_1}p_2^{n_2}}$, where $p_1$ and $p_2$ are distinct primes and $n_i\geq 2$.
			\item $G\cong P_1\rtimes P_2$, where $P_1$ is an elementary abelian Sylow $p_1$-group of order $p_1$ and $P_2$ is a Sylow subgroup of $G$ of prime order $p_2$ with $p_2|(p_1-1)$.
		\end{enumerate}
	\end{enumerate}
\end{utheorem}

In our final main result, we characterize the family of finite groups whose subgroup based generating graphs are perfect.  Recall that a graph \(G\) is said to be perfect if, for every induced subgraph \(H\) of \(G\), the chromatic number of \(H\) equals its clique number. To state our final our final main result, we first introduce the following definitions.

 A subgroup $H$ of a finite group $G$ is a maximal intersection if $H=\bigcap_{i\in \Lambda}m_i$, $\Lambda$ is an index set and $m_i\in \mathcal{M}(G)$  for each $i\in \Lambda$ \cite{lucchini3}. Let $\mathcal{S_P(G)}$ denote the set consisting of $G$ together with all maximal intersection subgroups of $G$. Then \(\mathcal{S}_P(G)\) forms a subposet of the subgroup lattice of $G$.

\begin{utheorem}\label{perfectg}
	Let $G$ be a finite group. Then $\mathcal{G}(G)$ is perfect if and only if $G$ is isomorphic to one of the following groups:
	\begin{enumerate}[(a)]
		\item $G\cong \mathbb{Z}_{p^k}$, $\mathbb{Z}_{{p_1}^{n_1}{p_2}^{n_2}}$, $\mathbb{Z}_{{{p_1}^{n_1}}{{p_2}^{n_2}}{{p_3}^{n_3}}}$,
		\item $G$ is a $2$-group, 
		\item $G$ has exactly four maximal subgroups,
		\item $G\cong \mathbb{Z}_{p^s}\times \mathbb{Z}_{p^t}$, where $s+t=k$ such that $\vert G\vert=p^k$ or $G$ is a $2$-generated non abelian group of order $p^k$,
		\item $G\cong P_1\rtimes P_2$, where $P_1$ is an elementary abelian sylow $p_1$-group of order $p_1$ and $P_2$ is a Sylow subgroup of $G$ of prime order $p_2$, and $p_2|(p_1-1)$.
		\item  $G$ is a finite group with $\IN(G)=2 <\UIN(G)\leq 5$ and there does not exist $p_i\in \mathcal{S_P(G)}$ such that they form an induced cycle of odd length at least five in either $\mathcal{G}(G)$ or $\overline{\mathcal{G}(G)}$.
	\end{enumerate}
\end{utheorem}

The rest of the paper is organized as follows: In Section~\ref{secui}, we establish a condition for a finite group \(G\) to satisfy \(\IN(G)=\UIN(G)\) in terms of the poset of maximal subgroups of \(G\). We also determine the uniform intersection number for several families of non-nilpotent groups, including dihedral groups $\operatorname{D}_{2n}$, dicyclic groups $\operatorname{Dic}_{n}$, and quasi dicyclic groups $\operatorname{QDic}_n$. Section~\ref{secGg} is devoted to the subgroup-based generating graph, where we prove Theorems \ref{finiteabedia}, \ref{corfggirth}, \ref{thmcmg}, and \ref{perfectg}. In the penultimate section of the paper, we study the subgroup-based generating graph of $\mathbb{Z}_n$ and determine its minimum and maximum degrees, and number of pendant vertices. Finally, we classifying the cyclic group $\mathbb{Z}_n$ for which $\mathcal{G}(\mathbb{Z}_n)$ is Hamiltonian.

\section{Uniform intersection number of a group}\label{secui}

In this section, we establish some basic properties of the uniform intersection number of finite groups and determine this invariant for several classes of groups. We begin by relating the uniform intersection number of the direct product of two finite groups $G_1$ and $G_2$ to the uniform intersection numbers of $G_1$ and $G_2$.


\begin{proposition}\label{propprodphi}
    Let $G_1$ and $G_2$ be two finite groups. Then $\UIN(G_1\times G_2)\leq \UIN(G_1)+ \UIN(G_2)$.
\end{proposition}
\begin{proof}
  Let $\UIN(G_1)=k$ and $\UIN(G_2)=l$. Then $\bigcap_{i=1}^km_i=\Phi(G_1)$ and $\bigcap_{j=1}^ln_j=\Phi(G_2)$, where $m_i\in \mathcal{M}(G_1)$ for $i\in [k]$ and $n_j\in \mathcal{M}(G_2)$ for $j\in[l]$. Notice that, for each $i\in [k]$, $m_i\times G_2$ is a maximal subgroup of $G_1\times G_2$ and, for each $i\in[l]$, $G_1\times n_i$ is a maximal subgroup of $G_1\times G_2$. Since $\Phi(G_1\times G_2)=\Phi(G_1)\times \Phi(G_2)$, we have 
\begin{equation*}
    \bigcap_{i=1}^k(m_i\times G_2)\cap \bigcap_{j=1}^l(G_1\times n_j)=\Phi(G_1)\times \Phi(G_2)=\Phi(G_1\times G_2).
\end{equation*}
Thus $\UIN(G_1\times G_2)\leq \UIN(G_1)+\UIN(G_2)$.
\end{proof}
Let $G=\prod_{i=1}^kG_i$ be a finite abelian group of order $n=\prod_{i=1}^kp_i^{t_i}$ with \(|G_i|=p_i^{t_i}\) and \(p_1,\ldots,p_k\) distinct primes. By Goursat's theorem (see \cite{pet}), every subgroup $H$ of $G$ can be written as $H=\prod_{i=1}^kH_i$, where \(H_i\) is a subgroup of \(G_i\) for each \(i\in[k]\). Therefore, we obtain the following result:
\begin{corollary}\label{corabliui}
Let $G_1$ and $G_2$ be two finite abelian groups such that $\gc(|G_1|, |G_2|)=1$. Then $\IN(G_1\times G_2)=\IN(G_1)+ \IN(G_2)$ and $\UIN(G_1\times G_2)=\UIN(G_1)+ \UIN(G_2)$.
\end{corollary}

We next show that uniform intersection number of a finite group and its quotient are same under a suitable condition. To proceed we need a basic result.

\begin{lemma}\cite[Theorem $1$]{aiva} \label{lmfrat}
	Let $N$ be a normal subgroup of a finite group $G$ such that $N\subseteq \Phi(G)$. Then $\Phi(G)/N=\Phi\left(G/N\right)$.
\end{lemma}
\begin{lemma}
    Let $N$ be a normal subgroup of a finite group $G$. If $N\subseteq \Phi(G)$, then $\UIN({G}/{N})=\UIN(G)$. 
\end{lemma}
\begin{proof}
   Let $\UIN(G)=k$. Since $N\subseteq \Phi(G)$, there is a one to one correspondence between the set of maximal subgroups of $G$ and ${G}/{N}$. Using Lemma \ref{lmfrat}, we have \[\bigcap_{i=1}^k\left({m_i}/{N}\right)={\bigcap_{i=1}^km_i}/{N}={\Phi(G)}/{N}=\Phi({G}/{N})\]
   for any distinct $k$-maximal subgroups $m_1,\ldots,m_k$ of $G$.  Also note that there must exist $m_1$, $m_2, \dots , m_{k-1}\in \mathcal{M}(G)$ such that $\bigcap_{i=1}^{k-1}\left({m_i}/{N}\right)\not=\Phi({G}/{N})$, otherwise $\UIN(G)<k$, which contradicts our hypothesis. Hence, the result follows.
\end{proof}
We next determine the uniform intersection numbers of several non-nilpotent groups including, dicyclic, dihedral, and quasi-dihedral groups. The dicyclic group of order $4n$, denoted by $\dic_n$, is defined as 
\begin{equation*}
    \dic_n=\langle a,x: a^{2n}=1, ~x^2=a^n,~ x^{-1}ax=a^{-1}\rangle. 
\end{equation*}
Let $n=\prod_{i=1}^kp_i^{t_i}$. By \cite[Theorem 5.1]{hum}, $\IN(\dic_n)=k+1$. The following result determines the uniform intersection number of $\dic_n$.

\begin{theorem}\label{thmuidic}
Let $n=\prod_{i=1}^kp_i^{t_i}$ be with $p_1\geq p_2\geq \dots \geq p_k$. Then
 \[
\UIN(\dic_n)=\begin{cases}2~ &\text{if}~ k=1~\\ \sum\limits_{i=1}^{k-1}p_i +2~ &\text{if}~ k\geq 2~~
\end{cases}\]
\end{theorem}
\begin{proof}
  By \cite[Theorem 5.1]{hum}, the maximal subgroups of the group $\dic_n$ are of the form $\langle a\rangle$ and $\langle a^{p_j}, a^ix\rangle$, where $j\in [k]$ and $i\in \{0,1,\dots, p_j-1\}$. Thus $\Phi(\dic_n)=\langle a^{p_1\dots p_k}\rangle$ and $|\mathcal{M}(\dic_n)|=\sum_{i=1}^kp_i+1$. For each \(t\in[k]\), let \(C_{p_t}\) denote the conjugacy class of maximal subgroups of the form 
  \[
  \langle a^{p_t},a^ix\rangle,
  \qquad 0\leq i\leq p_t-1.
  \]
Clearly, $|C_{p_t}|=p_t$. If $k=1$, then the intersection of any two distinct maximal subgroups of $\dic_n$ is equal to the $\Phi(\dic_n)$ and hence $\UIN(\dic_n)=2$.

 Let $k\geq 2$. Then 
  \begin{equation*}
  	\left(\bigcap_{i=1}^{k-1}C_{p_i}\right)\cap \langle a\rangle=\langle a^{p_1p_2\cdots p_{k-1}} \rangle\not=\Phi(\dic_n)=\langle a^{p_1p_2\cdots p_k}\rangle.
  \end{equation*}
  Therefore, $\UIN(Dic_n)>\sum_{i=1}^{k-1}p_i+1$. Now, consider any collection of $\sum_{i=1}^{k-1}p_i+2$ maximal subgroups of $\dic_n$. Then this collection must contain at least one maximal subgroup from each conjugacy class. Hence, the intersection of any $\sum_{i=1}^{k-1}p_i+2$ distinct maximal subgroups of $\dic_n$ is $\Phi(G)$. 
  Therefore, $\UIN(\dic_n)=\sum_{i=1}^{k-1}p_i+2$.
\end{proof}

We now investigate the uniform intersection number of the dihedral group. The dihedral group of order $2n$, denoted by $\dih_{2n}$, is defined as 
\begin{equation*}
    \dih_{2n}=\langle r,s: r^n=s^2=1, srs=r^{-1}\rangle. 
\end{equation*}
In \cite[Theorem 5.4]{hum}, it was shown that $\IN(\dih_{2n})=k+1$.  In the following result, we determine the uniform intersection number of $\dih_{2n}$.
\begin{theorem}
 Let $n=\prod_{i=1}^kp_i^{t_i}$ be with $p_1\geq p_2\geq \dots \geq p_k$. Then
 \begin{align*}
\UIN(\dih_{2n})=\begin{cases}2~ &\text{if}~ k=1~\\ \sum\limits_{i=1}^{k-1}p_i+2~ &\text{if}~ k\geq 2~~
\end{cases}
\end{align*} 
\end{theorem}
\begin{proof}
We first assume that $k=1$. By \cite[Theorem 5.3]{hum}, the maximal subgroups of $\dih_{2n}$ are precisely $\langle r \rangle$ and $\langle r^{p_1},sr^j\rangle$, where $j\in \{0,1,2\dots,{p_1}-1\}$. Then $\Phi(\dih_{2n})=\langle r^{p_1}\rangle$ and the intersection of two distinct maximal subgroups of $\dih_{2p_1^{t_1}}$ is equal to $\langle r^{p_1}\rangle$. Thus $\UIN(\dih_{2n})=\IN(\dih_{2n})=2$.

We now prove the assertion for $k\geq 2$. By \cite[Theorem 5.4]{hum},  the maximal subgroups of $\dih_{2n}$ are of the form $\langle r \rangle$ and $\langle r^{p_i}, sr^j\rangle $, where $i\in [k]$  and $j\in \{0,1, \dots, p_{i-1}\}$. Thus $\Phi(\dih_{2n})=\langle r^{p_1p_2\cdots p_k}\rangle$. Let $C_{p_0}=\{\langle r \rangle\}$ and, for each $i\in [k]$, $C_{p_i}=\{\langle r^{p_i},sr^j\rangle: j\in [p_i-1]\}$. Then, by an argument analogous to that used in Theorem Theorem \ref{thmuidic}, we have $\UIN(\dih_{2n})=\sum_{i=1}^{k-1}p_i+2$.
\end{proof}

We next determine the uniform intersection number of the quasi-dihedral group and show that it coincides with its intersection number. Recall that, for \(n\geq 4\), the quasi-dihedral group \(   \qdic_n\) is defined by
\begin{equation*}
    \qdic_n=\langle r,s:r^{2^{n-1}}=s^2=1, srs=r^{2^{n-2}-1}\rangle.
\end{equation*}
\begin{theorem}
    Let $n\geq 4$ be an integer. Then $\UIN(   \qdic_n)=\IN(\qdic_n)=2$.
\end{theorem}
\begin{proof}
Note that \( \qdic_n\) is a group of order $2^{n}$ and the maximal subgroups of \( qdic_n\) are 
consist of exactly three isomorphism classes:
\begin{enumerate}[{(i)}]
    \item  The cyclic group $\mathbb{Z}_{2^{n-1}}$,
\item The dihedral group $\dih_{2^{n-1}}$,
\item The generalized quaternion group $\operatorname{Q}_{2^{n-1}}$. 
\end{enumerate}
Thus, every maximal subgroup of \(\qdic_n\) has order \(2^{n-1}\), and $\Phi(   \qdic_n)=\mathbb{Z}_{2^{n-2}}$. Moreover, for any two distinct maximal subgroups \(m_1,m_2\in\mathcal{M}(\qdic_n)\), $|m_1\cap m_2|=2^{n-2}=|\Phi(\qdic_n)|$.  Consequently, $\UIN(\qdic_n)=\IN(\qdic_n)=2$.
\end{proof}


We now show that for every finite cyclic group \(G\), the intersection number and the uniform intersection number coincide.
\begin{corollary}\label{uid-cyclic}
	Let $G$ be a nontrivial cyclic group of order $n=\prod_{i=1}^kp_i^{t_i}$. Then $\IN(G)=\UIN(G)=k$.
\end{corollary}

\begin{proof}
	By the Fundamental theorem of an abelian Group,
	\begin{align*}
		G\cong \mathbb{Z}_{{p_1}^{t_1}}\times \mathbb{Z}_{{p_2}^{t_2}}\times \cdots \times \mathbb{Z}_{{p_k}^{t_k}}.
	\end{align*}
	Since $\IN(\mathbb{Z}_{{p_j}^{t_j}})=1$ for each $j\in [k]$, by \cite[Theorem 4.7]{hum}, we have $\IN(G)=k$. Also, $k=\IN(G)\leq \UIN(G)\leq \vert \mathcal{M}(G)\vert=k$. Consequently, $\UIN(G)=k$.
\end{proof}

In Theorem \ref{thmcmg}, we characterize all finite groups $G$ for which $\IN(G)=\UIN(G)=2$, while the preceding proposition provides, for each \(k\), an example of a finite group \(G\) satisfying $\IN(G)=\UIN(G)=k$. However, a complete classification of all finite groups \(G\) such that  $\IN(G)=\UIN(G)=k>2$ appears to be an interesting and challenging problem. In the final result of this section, we give a necessary and sufficient condition for this equality in terms of the maximal intersection poset of a finite group. To this end, we first recall the following definitions.

 Let $M$ be a non empty subset of $\mathcal{M}(G)$. Then $M$ is irredundant if the intersection of the subgroups in $M$ is not equal to the intersection of any proper subset. The minimal dimension of $G$, denoted as $\operatorname{Mindim}(G)$, is the minimal size of a maximal irredundant set of $G$.
 \begin{proposition}
 	Let $G$ be a finite group with $|\mathcal{M}(G)|=m$ such that $\operatorname{Mindim}(G)=\IN(G)$. Then $\IN(G)=\UIN(G)=k$ if and only if $|\mathcal{S}_P(G)|=\sum\limits_{i=1}^{k-1}\left(^m_i \right)+2$, where $\mathcal{S_P(G)}$ is the maximal intersection poset of $G$.
 \end{proposition}
 \begin{proof}
 	Suppose $\IN(G)=\UIN(G)=k$. Then, for every $k$-element subset $\Lambda$ of  $\mathcal{M}(G)$, $\bigcap\limits_{m_i\in  \Lambda}m_i=\Phi(G)$. Consequently, $|\mathcal{S}_P(G)|=\sum_{i=1}^{k-1}\left(^m_i \right)+2$.
 	
 	Conversely, suppose $|\mathcal{S}_P(G)|=\sum_{i=1}^{k-1}\left(^m_i \right)+2$. Then $\UIN(G)=k$. It remains to show let $\IN(G)=k$. If not, then there exist maximal subgroups $m_1$, $m_2$,$\dots$, $m_t$ of $G$ such that $\bigcap\limits_{i=1}^tm_i=\Phi(G)$, where $t\leq k-1$. This implies $|\mathcal{S}_P(G)|<\sum_{i=1}^{k-1}\left(^m_i \right)+2$,  a contradiction. Thus $\IN(G)=\UIN(G)=k$. 
 \end{proof}

\section{Structure of subgroup-based generating graph}\label{secGg}

In this section, we study the subgroup-based generating graphs of finitely generated groups and prove our main results. We begin by investigating the connectedness and diameter of the subgroup-based generating graph of a  finitely generated group. We show that the diameter of this graph is at most \(3\) and explicitly determine its diameter for free abelian groups and finite abelian groups. In addition, we establish a relationship between the diameter of the generating graph of a finite group and its uniform intersection number. The following lemma will be useful in our study of connectedness and diameter.

\begin{lemma}\label{lmiff}
   Let $G$ be a finitely generated group. Then the two vertices $S_1$ and $S_2$ in $\mathcal{G}(G)$ are adjacent if and only if $M(S_1)\cap M(S_2)=\emptyset$.
\end{lemma}
\begin{proof}
	We prove the equivalent contrapositive statement:  the two vertices $S_1$ and $S_2$ in $\overline{\mathcal{G}(G)}$ are adjacent if and only if $M(S_1)\cap M(S_2)\not=\emptyset$.\smallskip
	
  Let $S_1\sim S_2$ in $\overline{\mathcal{G}(G)}$. Then $\langle S_1, S_2\rangle\not=G$. 
   Since $G$ is finitely generated, by Zorn's lemma, we have $\langle S_1, S_2\rangle \subseteq m$ for some $m\in \mathcal{M}(G)$. Then $m\in M(S_1)\cap M(S_2)$. Thus $M(S_1)\cap M(S_2)\not=\emptyset$.

  To prove the converse, suppose that there exist nontrivial proper subgroups \(S_1,S_2\leq G\), with \(S_1,S_2\nsubseteq \Phi(G)\), such that $M(S_1)\cap M(S_2)\not=\emptyset$. Then there exists $m\in \mathcal{M}(G)$ such that $m\in M(S_1)\cap M(S_2)$. This implies $\langle S_1, S_2\rangle \subseteq m\not=G$. Hence $S_1\sim S_2$ in $\overline{\mathcal{G}(G)}$.
\end{proof}

We next characterize the connectedness of the graph \(\mathcal{G}(G)\) associated with a finitely generated group \(G\). In fact, we establish the stronger result that, whenever \(\mathcal{G}(G)\) is connected, its diameter is at most \(3\).
\begin{proposition}\label{propdiam}
 Let $G$ be a finitely generated group. Then $\mathcal{G}(G)$ is a nontrivial connected graph if and only if $|\mathcal{M}(G)|\geq 2$. In this case, $\diam(\mathcal{G}(G))\leq 3$.
\end{proposition}
\begin{proof}
Let $\mathcal{G}(G)$ be a nontrivial connected graph. Since $G$ is finitely generated, we have $|\mathcal{M}(G)|\geq 1$. If $|\mathcal{M}(G)|=1$, then $\mathcal{G}(G)$ is a null graph, a contradiction. Consequently, $|\mathcal{M}(G)|\geq 2$.

Conversely, suppose that $|\mathcal{M}(G)|\geq 2$, and let $S_1, S_2$ be two arbitrary vertices of $\mathcal{G}(G)$. If $S_1\sim S_2$, there is nothing to prove. Thus, assume that $S_1\nsim S_2$. Then $M(S_1)\cap M(S_2)\not=\emptyset$. Since $S_1$, $S_2\nsubseteq \Phi(G)$, there exist $m_1$, $m_2\in \mathcal{M}(G)$ such that $S_1\nsubseteq m_2$ and $S_2\nsubseteq m_1$. Thus, \(S_1\sim m_2\) and \(m_1\sim S_2\). If $m_1=m_2=m$, then  $S_1\sim m\sim S_2$.  On the other hand, if \(m_1\neq m_2\),  the maximality of \(m_1\) and \(m_2\) yields $ \langle m_1,m_2\rangle=G$ and hence \(m_1\sim m_2\). Consequently,  $S_1\sim m_2\sim m_1\sim S_2$ is a $S_1$-$S_2$ path of length at most \(3\). Therefore, $\mathcal{G}(G)$ is  connected and $\diam(\mathcal{G}(G))\leq 3$.
\end{proof}

We now establish necessary and sufficient conditions for each possible value of the diameter of \(\mathcal{G}(G)\), where \(G\) is a finitely generated group.
\begin{theorem}\label{thmdiam}
    Let $G$ be a finitely generated group with $|\mathcal{M}(G)|\geq 2$. Then 
    \begin{enumerate}[{(1)}]
    	\item   $\diam(\mathcal{G}(G))=1$ if and only if $V(\mathcal{G}(G))= \mathcal{M}(G)$ .
    	\item  $\diam(\mathcal{G}(G))=2$ if and only if either $\vert V(\mathcal{G}(G))\setminus \mathcal{M}(G)\vert=1$, or $\vert V(\mathcal{G}(G))\setminus \mathcal{M}(G)\vert\geq 2$ and, for every pair of nonadjacent vertices $S_1$, $S_2\in V(\mathcal{G}(G))$,  $M(S_1)\cup M(S_2)\subsetneq \mathcal{M}(G)$.
    	
    	\item  $\diam(\mathcal{G}(G))=3$ if and only if  $|V(\mathcal{G}(G))\setminus \mathcal{M}(G)|\geq 2$ and there exist nonadjacent vertices $S_1$, $S_2\in V(\mathcal{G}(G))$ such that  $M(S_1)\cup M(S_2)=\mathcal{M}(G)$.
    \end{enumerate}
   
\end{theorem}
\begin{proof}
	Since $|\mathcal{M}(G)|\geq 2$, by Proposition \ref{propdiam}, $\mathcal{G}(G)$ is connected. 
\begin{enumerate}[{(1)}]
    \item Since $G$ is finitely generated, each subgroup of $G$ is contained in some maximal subgroups of $G$. Therefore, by Lemma \ref{lmiff}, we have $\diam(\mathcal{G}(G))=1$ if and only if every vertex of $\mathcal{G}(G)$ is a maximal subgroup of $G$.
\item Suppose that $\diam(\mathcal{G}(G))=2$.  By the first part, $\vert V(\mathcal{G}(G))\setminus \mathcal \mathcal{M}(G)\vert\geq 1$. If $\vert V(\mathcal{G}(G))\setminus \mathcal \mathcal{M}(G)\vert=1$, we are done. Let $\vert V(\mathcal{G}(G))\setminus \mathcal \mathcal{M}(G)\vert\geq 2$. It remains to show  for every pair of nonadjacent vertices $S_1$, $S_2\in V(\mathcal{G}(G))$,  $M(S_1)\cup M(S_2)\subsetneq \mathcal{M}(G)$. Suppose, to the contrary, that there exist nonadjacent vertices $S_1$, $S_2$ in $ V(\mathcal{G}(G))\setminus \mathcal{M}(G)$ such that $M(S_1)\cup M(S_2)=\mathcal{M}(G)$. Then there is no vertex $H\in V(\mathcal{G})$ for which $S_1\sim H\sim S_2$. Therefore, $\di(S_1, S_2)\geq 3$, a contradiction. Consequently, $M(S_1)\cup M(S_2)\subsetneq \mathcal{M}(G)$ for every pair of nonadjacent vertices \(S_1,S_2\in V(\mathcal{G}(G))\).

To prove the converse, first assume that $\vert V(\mathcal{G}(G))\setminus \mathcal \mathcal{M}(G)\vert=1$, and  let  $A\in V(\mathcal{G}(G))\setminus \mathcal{M}(G)$. Since $G$ is finitely generated and $A\nsubseteq \Phi(G)$, there exist $m_1$, $m_2\in \mathcal{M}(G)$ such that $A\subsetneq m_1$ and $A\nsubseteq m_2$. Thus $A\sim m_2\sim m_1$. Since \(A\subsetneq m_1\), the vertices \(A\) and \(m_1\) are nonadjacent, and hence $\di(A,m_1)$=2. Furthermore, any two distinct maximal subgroups of \(G\) are adjacent. It follows that $\diam(\mathcal{G}(G))=2$.

Now suppose that  $\vert V(\mathcal{G}(G))\setminus \mathcal \mathcal{M}(G)\vert\geq2$. By part $(1)$,  $\mathcal{G}(G)$ is not complete.  Let $S_1$, $S_2\in V(\mathcal{G}(G))$ be two non adjacent vertices.  By hypothesis, there exists $m\in \mathcal{M}(G)$ such that $m\notin M(S_1)\cup M(S_2)$. Thus, every pair of nonadjacent vertices of \(\mathcal{G}(G)\) is at distance \(2\), and consequently, $\diam(\mathcal{G}(G))=2$.

\item  This follows immediately from the preceding two cases, together with the fact that $\diam(\mathcal{G}(G))\leq 3$.
\end{enumerate}
\end{proof}
Recall that  Lu--Meng--Qin--Wei--Zhang \cite{lu} introduced the graph \(\Gamma^{*}(G)\), obtained by deleting the isolated vertices from the generating graph of a finite group \(G\), and characterized the finite groups for which \(\Gamma^{*}(G)\) is complete. For a finite group \(G\), the graph \(\Gamma^{*}(G)\) coincides with the graph \(\mathcal{G}(G)\) introduced in this paper. We now recall their characterization of the finite groups for which \(\mathcal{G}(G)\) is complete.
\begin{theorem}\cite[Theorem 2.5]{lu}\label{lmcom}
Let $G$ be a finite group. Then $\mathcal{G}(G)$ is complete if and only if $G$ is isomorphic to $\mathbb{Z}_{p^2}$, $\mathbb{Z}_p\times \mathbb{Z}_p$, $\mathbb{Z}_{pq}$, $Z_p\rtimes Z_q$ or the quaternion group $\operatorname{Q}_8$.
\end{theorem}

Let $G$ be a finite group with exactly two maximal subgroups. Then $\diam(\mathcal{G}(G))=1$ if and only if $G\cong \mathbb{Z}_{pq}$; otherwise, $\diam(\mathcal{G}(G))=2$. So, we next consider finite groups \(G\) with at least three maximal subgroups and determine the diameter of \(\mathcal{G}(G)\).

\begin{proposition}\label{uin-diam}
    Let $G$ be a finite group with $3\leq |\mathcal{M}(G)|=k$. Then

\begin{enumerate}[(i)]
	\item If $\UIN(G)\geq \IN(G)> \lceil \frac{k+1}{2}\rceil$, then $\diam(\mathcal{G}(G))=3$. 
	\item  If $\IN(G)\leq \UIN(G)\leq \lceil \frac{k+1}{2}\rceil$ and $G$ is not isomorphic to any of the groups listed in Theorem \ref{lmcom}, then $\diam(\mathcal{G}(G))=2$.
\end{enumerate}
\end{proposition}
\begin{proof}
$(i)$. Suppose that $\IN(G)>\lceil \frac{k+1}{2}\rceil$. Define $S_1:=m_1\cap m_2\cap \ldots \cap m_{\lceil\frac{k+1}{2}\rceil}$, $S_2:=m_{\lceil\frac{k+1}{2} \rceil}\cap \cdots \cap m_k$, where $m_i\in \mathcal{M}(G)$ for $i\in [k]$. Then $S_1,S_2 \in V(\mathcal{G}(G))$ and $m_{\lceil\frac{k+1}{2}\rceil}\in M(S_1)\cap M(S_2)$. Thus, the vertices \(S_1\) and \(S_2\) are not adjacent in $\mathcal{G}(G)$ and  $M(S_1)\cup M(S_2)=\mathcal{M}(G)$.  Therefore, by Theorem~\ref{thmdiam}, $\diam(\mathcal{G}(G))=3$.

$(ii)$. Since $G$ is not isomorphic to one of the groups mentioned in Lemma \ref{lmcom}, $\diam(\mathcal{G}(G))\geq 2$. We now consider two cases. First, we assume that  $\UIN(G)<\lceil \frac{k+1}{2}\rceil$.  Then for any $H\in \mathcal{G}(G)$, we have $|M(H)|\leq \lceil \frac{k+1}{2}\rceil -2$, since $H \nsubseteq \Phi(G)$. Therefore, for any $H_1$, $H_2\in V(\mathcal{G}(G))$, we have $|M(H_1)|+|M(H_2)|\leq 2\lceil \frac{k+1}{2}\rceil -4<k$. This implies $M(H_1)\cup M(H_2)\subsetneq \mathcal{M}(G)$ and by Theorem \ref{thmdiam}, $\diam(\mathcal{G}(G))=2$.

 Suppose that $\IN(G)\leq \UIN(G)=\lceil \frac{k+1}{2}\rceil$. Then, for any $H\in V(\mathcal{G}(G))$, we have $|M(H)|\leq \lceil\frac{k+1}{2}\rceil-1$. Let \(H_1,H_2\in V(\mathcal{G}(G))\) be two nonadjacent vertices. Since \(H_1\nsim H_2\), we have $M(H_1)\cap M(H_2)\neq\emptyset.$
  Therefore $|M(H_1)|+|M(H_2)|\leq 2\lceil \frac{k+1}{2}\rceil -3<k$ and hence $M(H_1)\cup M(H_2)\subsetneq \mathcal{M}(G)$. It now follows from Theorem~\ref{thmdiam} that $\diam(\mathcal{G}(G))=2$.
\end{proof}
As an immediate consequence of the preceding results, we obtain the diameter of $\mathcal{G}(G)$ when $G$ is cyclic.
\begin{corollary}
Let $G$ be a cyclic group of order $n=\prod_{i=1}^kp_i^{t_i}$, where $k\geq 2$. Then 
     \begin{align*}
\diam (\mathcal{G}(G))=\begin{cases}$1$~ &\text{if}~ k=2 \text{ and } t_1=t_2=1,  ~\\ $2$~ &\text{if}~ k=2 \text{ and either } t_1\geq 2 \text{ or } t_2\geq 2,~\\ 
3~ &\ k\geq 3.~
\end{cases}
\end{align*} 
\end{corollary}
\begin{proof}
	The result follows immediately from Corollary ~\ref{uid-cyclic} and Proposition ~\ref{uin-diam}.
\end{proof}

In the next two corollaries, we show that $\mathcal{G}(G)$ has diameter $2$ when $G$ is an almost simple group or a symmetric group.

\begin{corollary}
 Let $G$ be a finite and almost simple group with $\vert \mathcal{M}(G)\vert \geq 8$. Then $\diam(\mathcal{G}(G))=2$.  
\end{corollary}
\begin{proof}
Since $G$ is an almost finite simple group, by \cite[Theorem 1(i)]{burn}, we have $\IN(G)\leq 4$. By hypothesis $\vert \mathcal{M}(G)\vert \geq 8$. Consequently, 
\begin{equation*}
\IN(G)\leq 4 < \lceil \frac{8+1}{2}\rceil =5.
\end{equation*}
Therefore, by Proposition \ref{uin-diam}, $\diam(\mathcal{G}(G))=2$.
\end{proof}

\begin{corollary}
	Let $S_n$ be the symmetric group on the $n$ elements, where $n\geq 3$. Then
	\begin{align*}
		\diam (\mathcal{G}(S_n))=\begin{cases}$1$~ &\text{if}~ n=3  ~\\ $2$~ &\text{if}~ n\geq 4 .~
		\end{cases}
	\end{align*} 
\end{corollary}
\begin{proof}
	Let $n=3$. Then $S_n\cong \mathbb{Z}_3\rtimes \mathbb{Z}_2$. By Theorem \ref{lmcom}, $\diam(\mathcal{G}(G))=1$. Let $n\geq 4$. By \cite[Theorem 1]{burn}, $\IN(S_n)\leq 3$. Since for $n\geq 4$, $\vert \mathcal{M}(G)\vert\geq 8$ and
	\begin{equation*}
		\IN(S_n)\leq 3< \left\lceil \frac{k+1}{2}\right\rceil,
	\end{equation*}
	where $k=\vert \mathcal{M}(G)\vert$. One can observe that for $n\geq 4$, $S_n$ is not isomorphic to any of the groups mentioned in Theorem \ref{lmcom}. Therefore, from Proposition \ref{uin-diam}, we have $\diam(\mathcal{G}(S_n))=2$.
\end{proof}

In the next result, we show that the diameter of the subgroup based generating graph of a finitely generated free abelian group is two. 
\begin{theorem}
Let $G$ be a  finitely generated free abelian group. Then $\diam(\mathcal{G}(G))=2$. 
\end{theorem}
\begin{proof}
Let $S_1$, $S_2\in V(\mathcal{G}(G))$ be two nonadjacent vertices. Since $G$ is a finitely generated free abelian group, $G\cong \bigoplus_{i\in \Lambda}\mathbb{Z}$, for some finite index set \(\Lambda\). If $|\Lambda|=1$, then it follows from Theorem~\ref{thmdiam} that $\diam(\mathcal{G}(G))=2$. Hence, assume that $|\Lambda|=t\geq 2$, and let \(\{e_i\}_{i\in\Lambda}\) denote the standard basis of $G$.
Let $x=(x_1,x_2,\dots,x_t)\in S_1$ and $y=(y_1, y_2,\dots,y_t)\in S_2$ be two nonzero elements of $G$. If $x_j\not=0$ and $y_j\not=0$ for some $j\in [t]$, then $(x)\sim m\sim (y)$, where $m=\langle e_1, e_2,\dots,e_{j-1},(5x_jy_j+1)e_j, e_{j+1},\dots, e_t\rangle \not=G$. If not, then there must exist $i,j\in [t]$ such that $x_i\not=0$ and $y_j\not=0$. Define $m:=\langle e_1,e_2, \dots ,e_{i-1},e_i+x_iy_je_j,e_{i+1},\dots, e_{j-1}, e_j+x_iy_je_i,e_{j+1},\dots, e_t\rangle $.  We next show that $m \neq G$.  Let \(A\) be the \(t\times t\) integer matrix whose columns are the  generators of \(m\). Then the matrix \(A\) differs from the identity matrix only in the \(i\)-th and \(j\)-th coordinates.  Thus $\det(A)=1-(x_i y_j)^2$. Since \(x_i y_j\neq0\), it follows that \(\det(A)\neq\pm1\).  Now, suppose, for contradiction, that \(m=G\). Then the columns of \(A\) generate $\mathbb{Z}^t$ , and hence \(A\in \operatorname{GL}_t(\mathbb{Z})\). Consequently, \(\det(A)=\pm1\), a contradiction. Thus $m \neq G$ and one can easily observe that $(x)\sim m\sim (y)$. Hence $\diam(\mathcal{G}(G))=2$.


\end{proof}
We now prove Theorem \ref{finiteabedia},  which classifies finite abelian groups in terms of the diameter of their subgroup based generating graphs.
%
\begin{proof}[Proof of Theorem \ref{finiteabedia}]
    Let $G$ be an abelian group of order $n=\prod_{i=1}^kp_i^{t_i}$. Then $G\cong \prod_{i=1}^kS_i$, where each $S_i$ is the Sylow subgroup of order $p_i^{t_i}$. By the fundamental theorem of finite abelian groups, each \(S_i\) can be written as
   \[
    S_i\cong \prod_{j=1}^{n_i}\mathbb Z_{p_i^{r_{ij}}}, \quad \text{such that}
    \quad
    \sum_{j=1}^{n_i}r_{ij}=t_i.
    \] 
    Moreover, every subgroup $H$ of $G$ decomposes as $H=\prod_{i=1}^kH_i$, where $H_i$ is a subgroup of $S_i$ for each \(i\). We divide the proof into the following three cases.
    
 \noindent\textbf{Case 1.} $k\geq 3$: Let $H_1=S_1\times 0\times\cdots\times 0$
 and
  $H_2=0\times S_2\times 0\times\cdots\times 0$.  Then \(H_1,H_2\in V(\mathcal{G}(G))\). Since $H_1\nsim H_2$, we have  $\di(H_1,H_2)\geq 2$. We next show that $\di(H_1,H_2)\geq 3$.  Suppose, for contradiction, that there exists a proper subgroup $T=\prod_{i=1}^kT_i$ of $G$ such that $H_1\sim T\sim H_2$, where $T_i$ is a subgroup of $G_i$ for $i\in [k]$. Since $\langle H_1,T\rangle=G$, we must have $T_i=S_i$ for all $i\in \{2,\dots,k\}$. Similarly, from $\langle T,H_2\rangle=G$, it follows that $T_1=S_1$. Thus $T=G$, contradicting the assumption that \(T\) is a proper subgroup of \(G\). Hence $\di(S_1, S_2)\geq 3$. By Theorem \ref{thmdiam}, we have $\diam(\mathcal{G}(G))=3$.\medskip

\noindent\textbf{Case 2.} $k=2$: In this case, $G\cong S_1\times S_2$, where $S_1=\prod_{j=1}^{n_1}\mathbb{Z}_{{p_1}^{t_{1j}}}$ and $S_2=\prod_{j=1}^{n_2}\mathbb{Z}_{{p_2}^{t_{2j}}}$. First, suppose that $n_1=n_2=1$.  If $t_{11}=t_{21}=1$, then $G\cong \mathbb{Z}_{p_1}\times \mathbb{Z}_{p_2}$.  By Theorem \ref{lmcom}, $\diam(\mathcal{G}(G))=1$.

 Now suppose that \(t_{i1}\geq 2\) for some \(i\in[2]\). Then $\mathcal{M}(G)=\{m_1=\langle p_1\rangle \times \mathbb{Z}_{p_2}^{t_{21}}, m_2=\mathbb{Z}_{p_1}^{t_{11}}\times \langle p_2\rangle  \}$.  Therefore, $\mathcal{G(G)}$ is a complete bipartite graph with bipartition $V_1=\{H\in V(\mathcal{G}(G)): H\leq m_1,\ H\nleq m_2\}$ and $V_2=\{H\in V(\mathcal{G}(G)): H\leq m_2,\ H\nleq m_1\}$. Hence $\diam(\mathcal{G}(G))=2$.\smallskip
    
 Finally, suppose that \(n_i\geq 2\) for some \(i\in[2]\). Without loss of generality, assume that $n_1\geq 2$. Let $H_1=T_1\times 0$ and $H_2=0\times S_2$ be vertices of $V(\mathcal{G}(G))$, where $T_1$ is a proper subgroup of $S_1$. Since $H_1\nsim H_2$, $\di(H_1,H_2)\geq 2$. Suppose, for contradiction, that $\di(H_1,H_2)=2$. Then there exists $L=L_1\times L_2\in V(\mathcal{G}(G))$ such that $H_1\sim L\sim H_2$, where $L_i$ is a subgroup of $S_i$. Since $H_1\sim L$, we have $L_2=S_2$. On the other hand, $H_2\sim L$ implies $L_1=S_1$. Therefore, $L=S_1\times S_2=G$, a contradiction. Consequently, $\di(H_1,H_2)\geq 3$. Since $\diam(\mathcal{G}(G))\leq 3$, we have $\diam(\mathcal{G}(G))=3$. 
\medskip

\noindent\textbf{Case 3.} $k=1$: In this case  $G\cong \prod_{i=1}^{n_1}\mathbb{Z}_{p^{t_{1i}}}$ and the maximal subgroups are of the form $m_i=\mathbb{Z}_{p^{t_{11}}}\times \dots \times \langle p\rangle \times \dots  \mathbb{Z}_{p^{t_{1{n_1}}}}$. Therefore, $\Phi(G)=\langle p\rangle\times \langle p\rangle\times \dots \times \langle p\rangle$.  Now, if $n_1=1$, then \(G\) has a unique maximal subgroup, and hence $\mathcal{G}(G)$ is a null graph. Thus, assume that $n_1\geq 2$. \smallskip 

If  $n_1=2$ and  $t_{11}=t_{12}=1$, then $G\cong \mathbb Z_p\times\mathbb Z_p$ and Theorem \ref{lmcom} gives $\diam(\mathcal{G}(G))=1$. \smallskip 

Now suppose that either \(n_1\geq3\), or \(n_1=2\) and \(t_{11}t_{12}\neq1\). By Theorem~\ref{lmcom}, $\diam(\mathcal{G}(G))\geq2$. Let $S_1, S_2\in V(\mathcal{G}(G))$ be such that $S_1\nsim S_2$.  Since \(S_1,S_2\nsubseteq\Phi(G)\), choose $x=(x_1,\ldots,x_{n_1})\in S_1\setminus\Phi(G)$ and $y=(y_1,\ldots,y_{n_1})\in S_2\setminus\Phi(G).$ Thus, at least one coordinate of each of \(x\) and \(y\) is not divisible by \(p\).  If there exists an index \(r\) such that
$ p\nmid x_r$ and $p\nmid y_r,$ define $ \varphi:G\longrightarrow\mathbb Z_p $ by $\varphi((a_i)):=a_r\pmod p.$ Then
$
\varphi(x)\neq 0$ and $\varphi(y)\neq 0$.
Otherwise, choose indices \(r\) and \(s\) such that $ p\nmid x_r$ and $p\nmid y_s,$. Since no coordinate is simultaneously nonzero modulo \(p\) for both \(x\) and \(y\), we have $ p\mid y_r$ and $p\mid x_s$. Define $ \varphi :G\longrightarrow\mathbb Z_p $ by $\varphi((a_i))=a_r+a_s\pmod p.$ Then again
$
\varphi(x)\neq 0$ and $\varphi(y)\neq 0$. In either case, let
$
m=\ker\varphi.$  Since \(\varphi\) is surjective onto \(\mathbb Z_p\), \(m\) is a maximal subgroup of \(G\). Moreover, $S_1\nsubseteq m$ and 
$S_2\nsubseteq m$. Thus $S_1\sim m\sim S_2.$ Therefore every pair of nonadjacent vertices of \(\mathcal{G}(G)\) is at distance \(2\), and so $\diam(\mathcal{G}(G))=2$.  
\end{proof}

\subsection{Complete multipartite subgroup based generating graph}

In this subsection, we determine the girth of the subgroup-based generating graph of a finitely generated group and characterize those groups for which this graph is complete multipartite. We also prove that, for a finite group \(G\), the subgroup-based generating graph \(\mathcal{G}(G)\) is complete multipartite if and only if both the intersection number and the uniform intersection number of \(G\) are equal to \(2\).

We begin by determining the girth of the subgroup-based generating graph of a finitely generated group.
\begin{proposition}\label{propgirthfgc}
    Let $G$ be a finitely generated group. Then 
   \begin{align*}
\gr (\mathcal{G}(G))=\begin{cases}$3$~ &\text{if $|\mathcal{M}(G)|\geq 3$,}~\\ $4$~ &\text{if $|\mathcal{M}(G)|=2$ with for each $m\in \mathcal{M}(G)$ there exists a subgroup $v\subsetneq m$ of $G$}~\\ &\text{such that $v\nsubseteq \Phi(G)$},  
~\\ 
\infty~ &\text{otherwise}.~
\end{cases}
\end{align*} 
\end{proposition}

\begin{proof}
Suppose that \(|\mathcal{M}(G)|\geq 3\), and choose three distinct maximal subgroups \(m_1,m_2,m_3\in\mathcal{M}(G)\). Since any two distinct maximal subgroups of \(G\) generate \(G\),  $m_1\sim m_2\sim m_3\sim m_1$ forms a triangle inn $\mathcal{G}(G)$. Thus $\gr(\mathcal{G}(G))=3$.\smallskip

Now, suppose that $\mathcal{M}(G)=\{m_1,m_2\}$ and there exist subgroups $S_1, S_2\in V(\mathcal{G}(G)) \setminus \mathcal{M}(G)$ such that $M(S_i)=\{m_i\}$ for $i\in[2]$. Then $S_1\sim m_2\sim m_1\sim S_2\sim S_1$ forms a $4$-cycle in $\mathcal{G}(G)$. We claim that \(\mathcal{G}(G)\) contains no triangle. Suppose, to the contrary, that $H_1\sim H_2\sim H_3\sim H_1$ is a triangle. Since $H_1\sim H_2$, $M(H_1)\cap M(H_2)=\emptyset$. As \(\mathcal{M}(G)=\{m_1,m_2\}\), we may assume, without loss of generality, that $M(H_1)=\{m_1\}$ and $M(H_2)=\{m_2\}$. Again, since $H_1\sim H_3$ and $H_2\sim H_3$, we must have $M(H_1)\cap M(H_3)=\emptyset$ and $M(H_2)\cap M(H_3)=\emptyset$. Thus $M(H_3)=\emptyset$, which is impossible for a vertex corresponding to a proper subgroup of the finitely generated group \(G\). Therefore, \(\mathcal{G}(G)\) contains no triangle and we conclude that $\gr (\mathcal{G}(G))=4$.
 
Finally, suppose that, for some \(i\in[2]\), \(m_i\) is the only vertex \(S\in V(\mathcal{G}(G))\) satisfying $M(S)=\{m_i\}$.  Then \(\mathcal{G}(G)\) is a star graph, and hence $\gr(\mathcal{G}(G))=\infty$.
\end{proof}
As an immediate consequence of the preceding result, we obtain the following characterization of finitely generated groups \(G\) for which $\mathcal{G}(G)$ is a tree. In particular, $\mathcal{G}(G)$ is a tree if and only if it is a star.

\begin{corollary}\label{cortreestar}
Let $G$ be a finitely generated group such that $\mathcal{G}(G)$ is nontrivial. Then the following statements are equivalent.
\begin{enumerate}[{(i)}]
    \item $\mathcal{G}(G)$ is a tree.
\item $\mathcal{G}(G)$ is a star.
\item $|\mathcal{M}(G)|=2$ and $G$ has a maximal subgroup $m$ such that every proper subgroup of $m$ is contained in $\Phi(G)$. 
\end{enumerate}
\end{corollary}

With the above analysis on hand, we now prove Theorem \ref{corfggirth} -- classification of the family of finite groups in terms of the girth of the subgroup based generating graph. The proof relies on the following basic result.
\begin{lemma}\cite[Lemma $3.2$]{wei}\label{lmm2}
	Let $G$ be a finite group. Then $G$ has exactly two maximal subgroups if and only if $G$ is cyclic and $|G|$ is divisible by exactly two distinct primes.
\end{lemma}

\begin{proof}[Proof of Theorem \ref{corfggirth}]
 If $|\mathcal{M}(G)|\geq 3$, then Proposition \ref{propgirthfgc} yields $\gr(\mathcal{G}(G))=3$. Hence, assume that $|\mathcal{M}(G)|=2$. By Lemma \ref{lmm2}, we have $G\cong \mathbb{Z}_{{p_1^{t_1}}{p_2^{t_2}}}$. If $t_1$, $t_2\geq 2$, then $\langle p_1^2\rangle\sim \langle p_2\rangle\sim \langle p_1\rangle\sim \langle p_2^2\rangle\sim \langle p_1^2\rangle$ is a $4$-cycle in $\mathcal{G}(G)$. On the other hand, Proposition~\ref{propgirthfgc} implies that $\mathcal{G}(G)$ contains no triangles. Therefore, $\gr(\mathcal{G}(G)=4$.
 
  Finally, suppose that either $t_1=1$ or $t_2=1$. Without loss of generality, let $t_1=1$. Then $\mathcal{G}(G)\cong K_{1,|t_2|}$ and hence $\gr(\mathcal{G}(G))=\infty$. Combining all of the above cases, we have the desired result.
\end{proof}
As immediate consequences of Theorem~\ref{corfggirth}, we obtain the following two results, which classify, up to isomorphism, the groups \(G\) for which $\mathcal{G}(G)$ is a tree or a bipartite graph.
\begin{corollary}
	Let $G$ be a finite group. Then $\mathcal{G}(G)$ is a tree if and only if $G\cong \mathbb{Z}_{p_1^kp_2}$, where \(p_1\) and \(p_2\) are distinct primes and $k\in \mathbb{N}$. 
\end{corollary}

\begin{corollary}
    Let $G$ be a finite group. Then the following statements are equivalent.
    \begin{enumerate}[(i)]
    	\item $\mathcal{G}(G)$ is a bipartite graph.
    	\item  $\mathcal{G}(G)$ is a triangle-free graph.
    	\item $G\cong \mathbb{Z}_{p_1^{t_1}p_2^{t_2}}$, $p_1$ and $p_2$ are distinct primes and $t_1,t_2\in \mathbb{N}$.
    \end{enumerate}
\end{corollary}

We now turn our attention to proving Theorem \ref{thmcmg}. We first show that, for every finitely generated group $G$, the graph $\mathcal{G}(G)$ is multipartite.

\begin{proposition}\label{propmul}
Let $G$ be a finitely generated group with $2\leq |\mathcal{M}(G)|<\infty$. Then $\mathcal{G}(G)$ is a k-partite graph if and only if $|\mathcal{M}(G)|=k$. 
\end{proposition}
\begin{proof}
Suppose that \(\mathcal{G}(G)\) is a \(k\)-partite graph with partite sets \(V_1,\ldots,V_k\). Assume, for contradiction, that $|\mathcal{M}(G)|\geq k+1$. By the pigeonhole principle, some partite set, say \(V_t\), contains at least two maximal subgroups, say $m_1$ and $m_2$. Since $m_1\sim m_2$ in $\mathcal{G}(G)$, the partite set \(V_t\) is not independent, a contradiction. Thus, $|\mathcal{M}(G)|\leq k$. 

 We next show that $|\mathcal{M}(G)|\geq k$. If not,  then some partite set, say \(V_k\), contains no maximal subgroup. Let \(S\in V_k\). Since \(G\) is finitely generated, \(S\) is contained in a maximal subgroup \(m_j\), which lies in some partite set \(V_j\), \(j\neq k\). As \(V_j\) is independent and \(m_j\in V_j\), every vertex of \(V_j\) is contained in \(m_j\). Hence \(S\) may be moved to \(V_j\) without destroying independence. Repeating this for every \(S\in V_k\), we obtain a \((k-1)\)-partite representation of \(\mathcal{G}(G)\), contradicting the minimality of \(k\). Therefore, $|\mathcal{M}(G)|= k.$\smallskip

To prove the converse, suppose that  $\mathcal{M}(G)=\{m_1,\ldots,m_k\}$, where $2\leq k<\infty$. Then $\mathcal{G}(G)$ is a connected graph. Define 
 \begin{align*}
     \begin{split}
         V_1&:=\{H\in V(\mathcal{M}(G)): H\subseteq m_1\}\\
         V_2&:=\{H\in V(\mathcal{M}(G)): H\subseteq m_2~\text{and}~H\nsubseteq m_1\}\\
         \vdots
         \\
         V_t&:=\{H\in V(\mathcal{M}(G): H\subseteq m_t~\text{and}~v\nsubseteq m_i,\forall i\in [t-1] \},
         \\
         \vdots\\
         V_k&:=\{H\in V(\mathcal{M}(G)): H\subseteq m_k~\text{and}~H\nsubseteq m_i, \forall i\in [k-1]\}.
     \end{split}
     \end{align*}
     Clearly, $V_i\not=\emptyset$ for each \(i\in[k]\) and $\{V_i:i\in [k]\}$ form a partition of $V(\mathcal{G}(G))$. Since every \(V_i\) is an independent set, $\mathcal{G}(G)$ is a $k$-partite graph.
\end{proof}

We next provide a necessary and sufficient condition for $\mathcal{G}(G)$ to be a complete multipartite graph.
\begin{corollary}\label{corui}
    Let $G$ be a finitely generated group with $|\mathcal{M}(G)|=k\geq 2$. Then $\mathcal{G}(G)$ is a complete $k$-partite graph if and only if $m_1\cap m_2=\Phi(G)$ for all distinct $m_1$, $m_2\in \mathcal{M}(G)$.
\end{corollary}
\begin{proof}
We first assume that $\mathcal{G}(G)$ is a complete $k$-partite graph.  Let $m_1$, $m_2\in \mathcal{M}(G)$ such that $m_1\cap m_2\not=\Phi(G)$. Then $m_1\cap m_2\in V(\mathcal{G}(G))$ and $m_1\cap m_2$ must belong to some partite set. Since $m_1\sim m_2$, $m_1$ and $m_2$ belong to different partite sets. Moreover, $m_1\cap m_2$ is neither adjacent to $m_1$ nor $m_2$. This contradicts the fact that $\mathcal{G}(G)$ is a complete $k$-partite graph. Thus $m_1\cap m_2=\Phi(G)$ for all distinct $m_1$, $m_2\in \mathcal{M}(G)$. 

Conversely, let $\mathcal{M}(G)=\{m_1,\ldots,m_k\}$ and $m_i\cap m_j=\Phi(G)$ for all $i\neq j$. By Proposition \ref{propmul}, $\mathcal{G}(G)$ is a $k$-partite graph with partite sets $V_1$, $V_2$, $\dots$, $V_k$. Since any two maximal subgroups of $G$ are adjacent in $\mathcal{G}(G)$, without loss of generality assume that $m_i\in V_i$ for $i\in [k]$.  If a vertex $H\in V_i$ that is not adjacent to some vertex in $V_j$ for $i\neq j$, then $H\subseteq m_i\cap m_j=\Phi(G)$, which contradicts $H\in V(\mathcal{G}(G))$. Thus $\mathcal{G}(G)$ is a complete $k$-partite graph.
\end{proof}

We now prove Theorem \ref{thmcmg} -- characterize the finite groups whose subgroup-based generating graphs are complete multipartite. 

\begin{proof}[Proof of Theorem \ref{thmcmg}]
We show the equivalences (1)$\iff$(2) and (2)$\iff$(3), which together complete the proof. Note that (1)$\iff$(2) is immediate by Corollary \ref{corui}. To show (2)$\iff$(3), we spilt the proof into several propositions.
\end{proof}

\begin{lemma}\label{lmqmul}
Let $G$ be a finite group and $N$ be any proper normal subgroup of $G$. If $\mathcal{G}(G)$ is a complete multipartite graph, then $\mathcal{G}({G}/{N})$ is also a complete multipartite graph.
\end{lemma}
\begin{proof}
Let $N$ be a proper normal subgroup of $G$, and suppose that $\mathcal{G}(G)$ is complete multipartite. If $|\mathcal{M}(G/N)|\leq 1$, then $G/N$ is a null graph. Hence, assume that $|\mathcal{M}(G/N)|\geq 2$ and let ${m_1}/{N}$, ${m_2}/{N}\in \mathcal{M}(G/N)$ be distinct maximal subgroups. Then $m_1, m_2\in \mathcal{M}(G)$. Since $\mathcal{G}(G)$ is a complete multipartite graph, by Corollary \ref{corui}, $m_1\cap m_2=\Phi(G)$. Moreover, \(N\subseteq m_1\cap m_2\), and hence $N\subseteq\Phi(G).$ Consequently, 
\begin{equation*}
    m_1/N \bigcap m_2/N={m_1\cap m_2}/N=\Phi(G)/N.
\end{equation*}
By Lemma \ref{lmfrat},  $\Phi({G}/{N})={\Phi(G)}/{N}$. This implies $m_1/N\bigcap m_2/N=\Phi\left(G/N\right)$. By Corollary \ref{corui}, $\mathcal{G}({G}/{N})$ is a complete multipartite graph.

\end{proof}
\begin{lemma}\label{lmzppp}
Let $G$ be a group isomorphic to $\underbrace{\mathbb{Z}_p\times \mathbb{Z}_p\times \dots \times \mathbb{Z}_p}_{r ~\text{times}}$, where $p$ is a prime and $r\geq 3$ be integer. Then $\mathcal{G}(G)$ is not a complete multipartite graph. 
\end{lemma}
\begin{proof}
Suppose that $G\cong \underbrace{\mathbb{Z}_p\times \mathbb{Z}_p\times \dots \times \mathbb{Z}_p}_{r~\text{times}}$. Then $\Phi(G)=\{0\}$ and the intersection of the two maximal subgroups $m_1=0\times \mathbb{Z}_p\times \mathbb{Z}_p \dots\times \mathbb{Z}_p$ and $m_2=\mathbb{Z}_p\times 0\times\mathbb{Z}_p\times \dots \times \mathbb{Z}_p$ is not equal to $\Phi(G)$. By Corollary \ref{corui}, $\mathcal{G}(G)$ is not a complete multipartite graph.
\end{proof}

We now classify the $p$-group $G$ of order $p^n$ for which  $\mathcal{G}(G)$ is a complete multipartite.
\begin{proposition}\label{propcmppg}
Let $G$ be a $p$-group of order $p^n$, where $p$ is prime and $n\geq 2$ be a integer. Then $\mathcal{G}(G)$ is a complete multipartite graph if and only if $G\cong \mathbb{Z}_{p^s}\times \mathbb{Z}_{p^t}$, where $s+t=n$ or $G$ is a $2$-generated non abelian $p$-group.
\end{proposition}
\begin{proof}
Since $G$ be a $p$-group of order $p^n$, one of the following holds: $(i)$ $G\cong \mathbb{Z}_{p^n}$;	$(ii)$ $G\cong \mathbb{Z}_{p^s} \times \mathbb{Z}_{p^t}$ with $s+t=n$; $(iii)$ $G$ is an abelian $p$-group of rank $k\geq 3$; $(iv)$ $G$ is a nonabelian $p$-group. If $G\cong \mathbb{Z}_{p^n}$, then $G$ is a null graph.

Let $G\cong \mathbb{Z}_{p^s} \times \mathbb{Z}_{p^t}$, where $s+t=n$. Then $|\mathcal{M}(G)|=p+1$. Observe that $\Phi(G)\cong \mathbb{Z}_{p^{t-1}}\times\mathbb{Z}_{p^{s-1}}$.  Furthermore, for any two distinct maximal subgroups \(m_1,m_2\in\mathcal{M}(G)\), $m_1\cap m_2=\Phi(G)$. Consequently, $\mathcal{G}(G)$ is a complete $(p+1)$-partite graph.

Now consider the case where $G$ be an abelian $p$-group of rank $k\geq 3$. Then there exists a normal subgroup $H$ of $G$ such that $G/H$ is isomorphic to the elementary abelian group $\mathbb{Z}_p\times \mathbb{Z}_p\times \dots \times \mathbb{Z}_p$ ($k$ times). By Lemma \ref{lmqmul} and Lemma \ref{lmzppp}, $\mathcal{G}(G)$ is not a complete multipartite graph.\smallskip

Finally, suppose that $G$ be a non abelian $p$-group. Let $G$ be $2$-generated. Then $G/\Phi(G)\cong \mathbb{Z}_p\times \mathbb{Z}_p$, which is a $2$-dimensional vector space over $\mathbb{Z}_p$. Hence, there is a one to one correspondence between the maximal subgroups of $G$ and the one dimensional subspaces of $G/\Phi(G)$. Since any two distinct one-dimensional subspaces of \(G/\Phi(G)\) intersect trivially, it follows that $m_1\cap m_2=\Phi(G)$ for any two distinct maximal subgroups \(m_1,m_2\) of \(G\). By Corollary \ref{corui}, $\mathcal{G}(G)$ is complete multipartite.  On the other hand, if \(G\) requires \(k\geq3\) generator, then $G/\Phi(G)\cong \mathbb{Z}_{p}^k$. Therefore, by Lemma \ref{lmzppp}, $\mathcal{G}(G)$ is not a complete multipartite graph.
\end{proof}
	Since a complete characterization of \(2\)-generated non abelian \(p\)-groups is rather involved, we leave this problem for future investigation. We conclude with a few examples of \(2\)-generated non abelian \(p\)-groups.

\begin{example}
	\begin{enumerate}[(i)]
		\item Every non abelian group $G$ of order $p^3$ is $2$-generated.
		\item  Dihedral group of order $2^n$.
		\item Quasi-dihedral group of order $n\geq 4$.
		\item $G=\mathbb{Z}_{p^k}\rtimes \mathbb{Z}\cong\langle x,y: x^{p^k}=y^p=1, yxy^{-1}=x^{1+p^{k-1}}\rangle$ is a $2$-generated non abelian group of order $p^k$
	\end{enumerate}
\end{example}

We next classify  finite nilpotent group $G$ such that  $\mathcal{G}(G)$ is complete multipartite. The proof requires the following basic result.

		\begin{lemma}(\cite{robin}, 9.3.5)\label{lmpdf}
			Let $G$ be a finite group, and let $p$ be a prime divisor of $|G|$. Then $p$ divides $[G:\Phi(G)]$.
	\end{lemma}

\begin{proposition}\label{propnilcmg}
Let $G$ be a finite nilpotent group of order $n=\prod_{i=1}^kp_i^{n_i}$, where $k\geq 2$ and $p_1,\ldots,p_k$ are distinct primes. Then $\mathcal{G}(G)$ is complete multipartite if and only if $G\cong \mathbb{Z}_{p_1^{n_1}p_2^{n_2}}$.
\end{proposition}
\begin{proof}
First  suppose that $G\cong \mathbb{Z}_{p_1^{n_1}p_2^{n_2}}$. Then $\mathcal{M}(G)=\{\langle p_1\rangle, \langle p_2\rangle \}$. This implies $\mathcal{G}(G)$ is a complete bipartite graph with bipartition $V_1=\{H\in V(\mathcal{G}(G)):H\subseteq \langle p_1 \rangle \}$ and $V_2=\{H\in V(\mathcal{G}(G)): H\subseteq \langle p_2\rangle\}$.

Conversely, suppose that $\mathcal{G}(G)$ is complete multipartite. We distinguish two cases according as \(G\) is cyclic or noncyclic. Assume first that $G$ is a non-cyclic nilpotent group of order $n$ with $k\geq 2$. By Lemma \ref{lmpdf}, we have $|\pi({G}/{\phi(G)})|=|\pi(G)|\geq 2$, which implies ${G}/{\phi(G)}\cong \prod_{i=1}^kG_i$, where each $G_i$ is an elementary abelian group and $k\geq 2$. If $k\geq 3$, then by Lemmas \ref{lmqmul} and \ref{lmzppp}, $\mathcal{G}(G)$ is not a complete multipartite graph. Let $k=2$. Since $G$ is non cyclic, one of $G_i$ must be non cyclic. Without loss of generality suppose that $G_1$ is non cyclic. Then there exist maximal subgroups $m_1, m_2 \in \mathcal{M}(G_1)$ such that $m_1\cap m_2 \not=\Phi(G_1)$. Now consider the maximal subgroups $m_1\times G_2$ and $m_2\times G_2$ of $G_1\times G_2$. Since $\Phi (G_1\times G_2)=\Phi(G_1)\times \Phi(G_2)$,   $m_1\times G_2\cap m_2\times G_2\neq \Phi (G_1\times G_2)$. By Proposition \ref{propmul} and Lemma \ref{lmqmul}, $\mathcal{G}(G)$ is not a complete multipartite graph.\smallskip 

We now assume that \(G\) is cyclic nilpotent of order \(n\).
Then $G\cong \mathbb{Z}_n$ and $\mathcal{M}(G)=\{\langle p_i \rangle : i\in [k]\}$. If $k\geq 3$, then $\langle p_1 \rangle \cap \langle p_2 \rangle \not=\Phi(G)=\langle p_1 \rangle \cap \langle p_2 \rangle \cap \dots \cap \langle p_k \rangle =\langle p_1p_2\cdots p_k\rangle$,  contradicting Corollary \ref{corui}. Hence $k=2$ and $G\cong \mathbb{Z}_{p_1^{n_1}p_2^{n_2}}$.
\end{proof}

We finally classify finite non-nilpotent groups $G$ for which $\mathcal{G}(G)$ is complete multipartite. Together with the preceding three propositions, this completes the proof of $(2)\iff (3)$ in Theorem \ref{thmcmg}.
\begin{proposition}
Let $G$ be a non-nilpotent group. Then $\mathcal{G}(G)$ is a complete multipartite graph if and only if $G\cong P_1\rtimes P_2$, where $P_1$ is an elementary abelian sylow $p_1$-group of order $p_1$ and $P_2$ is a Sylow $p_2$-subgroup of $G$ of order $p_2$ with $p_2|(p_1-1)$.
\end{proposition}
\begin{proof}
Suppose that $G\cong P_1\rtimes P_2$, where $P_1$ is an elementary abelian sylow $p_1$-group of order $p_1$ and $P_2$ is a Sylow $p_2$-subgroup of $G$ of order $p_2$, and $p_2|(p_1-1)$. Then $P_1\cong \mathbb Z_{p_1}$ and $P_2\cong \mathbb Z_{p_2}$. Hence $G\cong \mathbb Z_{p_1}\rtimes \mathbb Z_{p_2},$ where \(p_1\neq p_2\). Since \(G\) is non-nilpotent, the action of \(\mathbb Z_{p_2}\) on \(\mathbb Z_{p_1}\) is nontrivial and  $\Phi(G)=\{1\}$. Thus $\mathcal{G}(G)$ is a complete multipartite graph.\smallskip
	
Conversely, suppose that $\mathcal{G}(G)$ be a complete multipartite graph. We divide the remainder of the proof into two cases.

\textbf{Case 1.} $\Phi(G)$ is trivial: Since \(G\) is non-nilpotent, there exists a non-normal maximal subgroup \(m\) of \(G\). As \(m\leq N_G(m)\leq G\) and \(m\) is maximal, the non-normality of \(m\) implies that  $m=N_G(m)$. Hence, for every \(x\in G\setminus m\), we have \(m^x\neq m\). Thus \(m\) and \(m^x\) are distinct maximal subgroups, and Corollary~\ref{corui} yields $m\cap m^x=\Phi(G)=\{1\}$. This implies $G$ is a Frobenius group with the Frobenius complement $m$. Hence $G\cong k_G\rtimes m$, where $k_G$ is the Frobenius Kernel of $G$ and $\gc(|m|,|k_G|)=1$.  We next show that the group $m$ contains only one maximal subgroup i.e., $|\mathcal{M}(m)|=1$. If not, then $|\mathcal{M}({G}/{k_G})|=|\mathcal{M}(m)|\geq 2$. By Corollary \ref{corui}, $k_G$ is a subgroup of $\Phi(G)$,  a contradiction. Therefore, $|\mathcal{M}(m)|=1$.
This implies $m\cong \mathbb{Z}_{p_2^t}$ for some prime $p_2$ and a natural number $t$. If $t\geq 2$, then $1\not=m_{p_2}=\{x\in m:x^{p_2}=1\}$ is a non trivial proper subgroup of $m$. Hence, $\langle k_G,m_{p_2}\rangle$ is a nontrivial proper subgroup of $G$ contained in a maximal subgroup $m_1\in \mathcal{M}(G)\setminus\{m\}$. Since $G$ is a complete multipartite, $m_{p_2}\subseteq m\cap m_1=\Phi(G)=\{1\}$, which is a contradiction. Therefore, $m\cong \mathbb{Z}_{p_2}$.

We now show that $k_G\cong \mathbb{Z}_{p_1}^r$ for some prime $p_1\not=p_2$. Since $k_G$ is nilpotent,it is isomorphic to the direct product product of its Sylow subgroups, i.e. $k_G\cong \prod_{i=1}^nS_i$. If possible, let $n\geq 2$. Then $m\subsetneq \langle S_1,m\rangle \not=G$, which contradicts the maximality of $m$. Therefore, $k_G$ is a Sylow $p_1$ subgroup of $G$ for some prime $p_1\not=p_2$ such that $p_2|p_1-1$. Also, since $G$ is solvable, $\Phi(k_G)=k_G\cap \Phi(G)=\{1\}$. Consequently, $k_G\cong \mathbb{Z}_{p_1}^r$, where $r$ is the rank of $r$.   If $r\geq 2$, then there exist maximal subgroups $k_1\rtimes \mathbb{Z}_{p_2}$ and $k_2\rtimes \mathbb{Z}_{p_2}$ such that $k_1\rtimes \mathbb{Z}_{p_2}\cap k_2\rtimes \mathbb{Z}_{p_2}\not=\{1\}$, which contradicts that is complete multipartite graph
Consequently, $G\cong \mathbb Z_{p_1}\rtimes \mathbb Z_{p_2}$.


\textbf{Case $2$.} $\Phi(G)$ is non trivial: Since $\mathcal{G}(G)$ be a complete multipartite graph, by Lemma \ref{lmqmul}, ${G}/{\Phi(G)}$ is a complete multipartite graph. If possible, let ${G}/{\Phi(G)}$ be a nilpotent group.  Since $G$ is non abelian, ${G}/{\Phi(G)}$ is non cyclic. By Lemma \ref{lmpdf}, every prime divisor of the order of a finite group also divides $[G:\Phi(G)]$. Thus, from Propositions \ref{propcmppg} and \ref{propnilcmg}), we must have $|\pi({G}/{\Phi(G)})|=|\pi(G)|=1$, which is a contradiction. Thus ${G}/{\Phi(G)}$ is not a nilpotent group. Similarly as above, we have ${G}/{\Phi(G)}\cong {P_1}/{\Phi(P_1)}\rtimes {P_2}/{\Phi(P_2)}\cong \overline{P_1}\times \overline{P_2}$, where $P_i$ is a Sylow $p_i$ subgroup of $G$ for $i\in \{1,2\}$. Thus $G=P_1P_2$. Since $P_1\Phi(G)$ is a normal subgroup of $G$, we have $G=N(P_1)$, the normalizer of $P_1$ in $G$ (\cite{robin}, 5.2.14). Hence $P_1$ is normal in $G$. Since $\overline{P_1}= \frac{P_1\Phi(G)}{\Phi(G)}\cong \frac{P_1}{P_1\cap \Phi(G)}=\frac{P_1}{\Phi(P_1)}\cong \mathbb{Z}_{p_1}$, we have $P_1\cong \mathbb{Z}_{p_1^s}$ for some $s\in \mathbb{N}$. Now, if possible, let $|\mathcal{M}(P_1)|\geq 2$. Since $\frac{G}{P_1}\cong P_2$, we have $|M(P_2)|\geq 2$. Also, since $\mathcal{G}(G)$ is a complete multipartite graph, $P_2\subseteq m_1\cap m_2=\Phi(G)$ ($cf.$ Corollary \ref{corui}), which is a contradiction. Therefore, $|\mathcal{M}(P_2)|=1$, i.e. $P_2\cong \mathbb{Z}_{p_2^t}$ for some $t\in \mathbb{N}$. 

If possible, let $s=t=1$. Then $\Phi(G)=\{1\}$, which contradicts our hypothesis. 
Let at least one of $s$ and $t$ be greater than one. Then $\Phi(G)\cong \mathbb{Z}_{{p_1}^{s-1}}\times Z_{p_2}^{t-1}$ and one can verify that there exist maximal subgroups whose intersection is not equal to the Frattini subgroup of $G$. This concludes the proof.
\end{proof}

A complete multipartite graph is said to be balanced if all its partite sets have the same cardinality. As an immediate consequence of Theorem~\ref{thmcmg}, we obtain the following result.
\begin{corollary}
Let $G$ be a finite group. Then $\mathcal{G}(G)$ is a  balanced complete multipartite graph if and only if $G$ is isomorphic to one of the following: 

\begin{enumerate}[(a)]
	\item $G\cong \mathbb{Z}_{{p}^t}\mathbb{Z}_{{q}^t}$ or $G\cong \mathbb{Z}_{p}\times \mathbb{Z}_p$, where $p$ and $q$ are distinct primes. 
	\item $G\cong \operatorname{Q}_8$ or $G\cong \langle a,b,c:a^p=b^p=c^p=1, a^c=a, b^c=b, a^b=ac\rangle$, where $p$ is an odd prime. 
	\item $G\cong P\rtimes Q$, where $P$ and $Q$ are groups of prime orders $p$ and $q$, respectively, such that $q|(p-1)$.
\end{enumerate}
\end{corollary}
\subsection{Perfectness and Forbidden subgraph}
In this subsection, we prove our final main result, Theorem \ref{perfectg}. We characterize finite groups whose subgroup-based generating graphs are perfect. To complete the proof, we use the Strong Perfect Graph Theorem of Chudnovsky--Robertson--Seymour \cite{chud}, together with the uniform intersection number of a finite group. We first recall the definition of Berge graph.
\begin{definition}
A graph \(G\) is a Berge if no induced subgraph of \(G\) is an odd cycle of length at least five or the complement of one.
\end{definition} 	
We now  recall the Strong Perfect Graph Theorem.
\begin{lemma}\cite[1.2]{chud}\label{lmspgt}
A graph \(G\) is perfect if and only if it is a Berge.
\end{lemma}
Before proving the main result, we shall prove the following lemmas which are required to complete the proof.
\begin{lemma}\label{propper}
   Let $G$ be a finitely generated group with at most four maximal subgroups. Then the graph $\mathcal{G}(G)$ is perfect.
\end{lemma}
\begin{proof}
By Lemma \ref{lmspgt}, $\mathcal{G}(G)$, it suffices to show that neither $\mathcal{G}(G)$ nor its complement $\overline{\mathcal{G}(G)}$ contains an induced odd cycle of length at least five.

Suppose, for contradiction, that $S:S_1\sim S_2\sim \cdots \sim S_{2k+1}\sim S_1$ is an induced cycle in $\mathcal{G}(G)$, where $k\geq 2$. First, we show that $|M(S_t)|\geq 2$ for every vertex $S_t$ of $S$. Suppose $|M(S_t)|=1$ for some $S_t\in S$. Then $M(S_t)=\{m\}$ for some $m\in \mathcal{M}(G)$. Since $S_t\nsim S_{t+2}$ in $S$, $m\in M(S_{t+2})$. Similarly $S_t\nsim S_{t+3}$ implies $m\in M(S_{t+3})$. Hence $m\in M(S_{t+2})\cap M(S_{t+3})$ which implies $S_{t+2}\nsim S_{t+3}$, a contradiction. Hence $M(S_t)\geq 2$ for any $t\in [2k+1]$. By hypothesis, $G$ has at most four maximal subgroups. Let $\mathcal{M}(G)=\{m_1,m_2,m_3,m_4\}$. Without loss of generality, let $ M(S_1)=\{m_1, m_2\}$. Since $S_1\sim S_2$, we must have $m_3$, $m_4\in M(S_2)$. Again, as $S_2\sim S_3$ in $S$, we must have $m_1, m_2\in M(S_3)$. Also, $S_1\sim S_{2k+1}$ and therefore, $m_3$, $m_4\in M(S_{2k+1})$. Since $M(S_3)\cap M(S_{2k+1})=\emptyset$, we must have $S_3\sim S_{2k+1}$, which contradicts that $S$ is an induced cycle. Hence, $\mathcal{G}(G)$ contains no induced odd cycle of length at least $5$. By a similar argument, the same conclusion holds when $|\mathcal{M}(G)|=2$ or $3$.

Now, suppose for contradiction that $S':S_1'\sim S_2'\sim \cdots \sim S_{2k+1}'\sim S_1'$ be an induced cycle in $\overline{\mathcal{G}(G)}$, where $k\geq 2$. To prove that $|M(S'_t)|\geq 2$ for any $t\in[2k+1]$, let $M(S'_t)=\{m\}$ for some $t\in [2k+1]$ and $m\in \mathcal{M}(G)$. Since $S'_t\sim S_{t+1}$ in $\overline{\mathcal{G}(G)}$, we must have $m\in M(S'_{t+1})$. Again, as $S'_t\sim S'_{t-1}$, we must have $m\in M(S'_{t-1})$. Therefore, $m\in M(S'_{t+1})\cap M(S'_{t-1})$ and therefore $S'_{t-1}\sim S'_{t+1}$ in $S'$. This contradicts the notion that $S'$ is an induced cycle in $\overline{\mathcal{G}(G)}$. Hence $|M(S'_t)|\geq 2$ for any $t\in [2k+1]$.
Let $\mathcal{M}(G)=\{m_1,m_2,m_3,m_4\}$ and without loss of generality, let $M(S'_1)=\{m_1,m_2\}$. Since $S'_1\nsim S'_3$ in $S'$, we must have $m_3$, $m_4\in M(S'_3)$. Again, as $S'_3\nsim S'_{2k+1}$, we must have $m_1$, $m_2\in M(S'_{2k+1})$. In addition, $S'_{2k+1}\nsim S'_2$ implies $m_3$, $m_4\in M(S'_2)$. Then $M(S'_1)\cap M(S'_2)=\emptyset$ and hence, they are not adjacent in $\overline{\mathcal{G}(G)}$, a contradiction. Similarly, we can show that if $\mathcal{M}(G)\leq 3$, then there does not exist any induced cycle of odd length of at least five in $\overline{\mathcal{G}(G)}$. 
\end{proof}

\begin{corollary}\label{corcmo}
	Let $C$ be an induced odd cycle of length at least five either in $\mathcal{G}(G)$ or in $\overline{\mathcal{G}(G)}$. Then $|M(S)|\geq 2$ for any $S\in V(C)$.
	\end{corollary}

\begin{lemma}\label{corperboth}
    Let $G$ be a finitely generated group. Then the graph $\mathcal{G}(G)$ is perfect if $|\mathcal{M}(G)|\leq 4$ or $|\mathcal{M}(G)|\geq 5$ and $m_1\cap m_2=\Phi(G)$ for all distinct $m_1$, $m_2\in \mathcal{M}(G)$.
\end{lemma}
\begin{proof}
If $|\mathcal{M}(G)|\leq 4$, then by Proposition $\ref{propper}$, $\mathcal{G}(G)$ is a perfect graph. Now, assume that $|\mathcal{M}(G)|\geq 5$ and $m_1\cap m_2=\Phi(G)$ for all distinct $m_1$, $m_2\in \Phi(G)$. If possible, let $C:C_1\sim C_2\sim \dots\sim C_{2k+1}\sim C_1$ be an induced cycle in $\mathcal{G}(G)$, where $k\geq 2$. By an argument similar to that used in the proof of Proposition \ref{propper}, every vertex of $C$ is contained in at least two maximal subgroups of $G$. Let \(C_i\in C\) and write $M(C_i)=\{m_1,\ldots,m_t\}$,
$t\geq 2$. Then $C_i\subseteq \bigcap_{i=1}^tm_i=\Phi(G)$, contradicting \(C_i\in V(\mathcal{G}(G))\). Hence, the result follows.
\end{proof}
\begin{lemma}\label{lmnperf}
Let $G$ be a finitely generated group with $|\mathcal{M}(G)|\geq 5$, and suppose that $m_i\cap m_j\not=\Phi(G)$ for all $m_i, m_j\in \mathcal{M}(G)$. Then $\mathcal{G}(G)$ is not a perfect graph.
\end{lemma}
\begin{proof}
Let $\mathcal{M}(G)=\{m_1,\ldots,m_t\}$, where $t\geq 5$. Then $m_1\cap m_2\sim m_3\cap m_4\sim m_1\cap m_5\sim m_2\cap m_3\sim m_4\cap m_5\cap m_1\cap m_2$ is an induced $C_5$ in $\mathcal{G}(G)$. Thus $\mathcal{G}(G)$ is not a perfect graph.
\end{proof}
Using the preceding results, we now prove Theorem \ref{perfectg}.
\begin{proof}[Proof of Theorem \ref{perfectg}]
  If $G$ is isomorphic to one of the groups in $(a)$, $(b)$ or $(c)$, then $G$ has at most four maximal subgroups. Hence by Corollary \ref{corperboth}, $\mathcal{G}(G)$ is a perfect graph. Now, suppose $G$ is isomorphic to one of the groups in $(d)$, $(e)$, then the uniform intersection number of $G$ is two. Therefore, Corollary~\ref{corperboth} again implies that \(\mathcal{G}(G)\) is perfect. Let $G$ be a group satisfying condition $(f)$. Suppose, to the contrary, that $C: S_1\sim S_2\sim S_3\sim \cdots \sim S_{2k+1}\sim S_1$ is an induced odd cycle in either $\mathcal{G}(G)$ or  $\overline{\mathcal{G}(G)}$, where $k\geq 2$. By Corollary \ref{corcmo}, we have $|M(S_i)|\geq 2$ for each $i$. Hence, for each $i$, there exists $p_i\in \mathcal{S}_{P}(G)$ such that $p_1\sim p_2\sim \cdots\sim p_{2k+1}\sim p_1$ form an induced odd cycle of length $2k+1$ in either
  $\mathcal{G}(G)$ or $\overline{\mathcal{G}(G)}$, a contradiction. Consequently $\mathcal{G}(G)$ is a perfect graph.

  Conversely, suppose that $G$ is a finite group such that $\mathcal{G}(G)$ is perfect. If $G$ has at most four maximal subgroups, then $G$ is isomorphic to one of groups listed in $(a)$, $(b)$ or $(c)$. Thus, assume that $|\mathcal{M}(G)|\geq 5$. If $\IN(G)\geq 3$, then by Lemma \ref{lmnperf}, $\mathcal{G}(G)$ is not perfect. Hence $\IN(G)=2$. We now consider three mutually exclusive cases according to the value of $\UIN(G)$. Suppose first that $\UIN(G)=2$. Then, by Theorem~\ref{thmcmg}, \(G\) is  isomorphic to one of the groups listed in \((d)\) or \((e)\). Next, suppose that $3\leq \UIN(G)\leq 5$. If there exist \(p_1,p_2,\ldots,p_{2k+1}\in\mathcal{S}_P(G)\), with \(k\geq2\), that form an induced odd cycle in either \(\mathcal{G}(G)\) or \(\overline{\mathcal{G}(G)}\), then, by Lemma \ref{lmspgt}, \(\mathcal{G}(G)\) is not perfect. Hence, no such induced odd cycle exists, and $G$ is one of the groups described in $(f)$. 
  
  Finally we show that  $\UIN(G)\geq 6$ is not possible. Indeed, if $\UIN(G)\geq 6$, then there exist five distinct maximal subgroups \(m_{i_1},m_{i_2},m_{i_3},m_{i_4},m_{i_5}\) of \(G\) such that $m_{i_1}\cap m_{i_2}\cap m_{i_3}\cap m_{i_4}\cap m_{i_5}
  \neq \Phi(G)$.  Consequently,  $m_{i_1}\cap m_{i_2}\sim m_{i_3}\cap m_{i_4}\sim m_{i_1}\cap m_{i_5}\sim m_{i_2}\cap m_{i_3}\sim m_{i_4}\cap m_{i_5}\sim m_{i_1}\cap m_{i_2}$ form an induced $C_5$ in $\mathcal{G}(G)$, a contradiction. This completes the proof.
\end{proof}
As an immediate consequence, we have the following.
\begin{corollary}
    The graph $\mathcal{G}(\mathbb{Z}_n)$ is perfect if and only if $n$ has at most four distinct prime divisors.
\end{corollary}



\section{Generating graph of a finite cyclic group}\label{seccyclic}
In this section, we study the subgroup-based generating graph of a finite cyclic group. We determine its maximum and minimum degrees, the number of pendant vertices, and characterize when it is Hamiltonian.

Let $n=\prod_{i=1}^{k}p_i^{t_i},$ where \(p_1,\ldots,p_k\) are distinct primes and \(t_i\geq1\). Since a finite cyclic group of order $n$ is isomorphic to the cyclic group $(\mathbb{Z}_n,+)$, we consider the graph $\mathcal{G}(\mathbb{Z}_n)$. 
Every subgroup of $\mathbb{Z}_n$ is generated by a divisor $d$ of $n$. Thus $|V(\mathcal G(\mathbb Z_n))|
=\prod_{i=1}^{k}(t_i+1)-\prod_{i=1}^{k}t_i-1.$  Let $H=\left\langle\prod_{i=1}^{t}p_i^{s_i}\right\rangle$, where $t\in[k-1]$ and $ 0\leq s_i\leq t$. Then $\deg(H)=\prod_{j:s_j=0}(s_j+1)-1$.
\begin{corollary}
    Let $x$, $y\in V(\mathcal{G}(\mathbb{Z}_n))$. Then $\deg(x)=\deg(y)$ if and only if
    \begin{equation*}
        \sum_{i\in [k]-\operatorname{in}(\pi(x))}(t_i+1)=\sum_{j\in [k]-\operatorname{in}(\pi(y))}(t_j+1),\end{equation*}
    where $\pi(x)$ is the set of all prime divisors of $x$ and $\operatorname{in}(\pi(x))$ denotes the set of indices corresponding to the prime divisors of $x$.
\end{corollary}
\begin{corollary}\label{cormm}
    Let $n=\prod_{i=1}^kp_i^{t_i}$, where $t_1\geq t_2\geq \dots\geq t_k$. Then the maximum and minimum degrees of $\mathcal{G}(\mathbb{Z}_n)$ are $\prod_{i=1}^{k-1}(t_i+1)-1$ and $t_k$, respectively.
\end{corollary}

\begin{proposition}\label{proppen}
    Let $n=\prod_{i=1}^kp_i^{t_i}$, where $t_1\geq t_2\geq \dots \geq t_k$ and $k\geq 2$. Then the following holds:
    \begin{enumerate}[(i)]
    	\item  $\mathcal{G}(G)$ has a pendant vertex if and only if  \(t_i=1\) for some \(i\in[k]\).
    	\item  Suppose that $t_i\geq 2 \quad \text{for } i\in[l],$ and 
    	$ t_i=1 \quad \text{for } i\in\{l+1,\ldots,k\}$.The number of pendant vertex is $(k-l)\prod_{i=1}^lt_i$.
    \end{enumerate} 
    Moreover, if \(t_k=1\), any vertex of the form $(\prod_{i=1}^{k-1}p_i^{s_i})$ with $s_i\in[t_i]$, is a pendant vertex of $\mathcal{G}(G)$.
\end{proposition}
\begin{proof}
  $(i)$ Let $P$ be a pendant vertex in $\mathcal{G}(\mathbb{Z}_n)$. Since $P\in V(\mathcal{G}(\mathbb{Z}_n))$, we have \(P\nleq \Phi(\mathbb{Z}_n)\). Hence there exists a prime divisor \(p_i\) of \(n\) such that $\langle p_i\rangle \notin M(P)$. If $t_i\geq 2$, then both $\langle p_i\rangle $ and $\langle p_i^2\rangle $ are adjacent to $P$, which is a contradiction. Hence  $t_i=1$.

  Conversely, suppose that $t_k=1$. Then every vertex of the form $(\prod_{i=1}^{k-1}p_i^{s_i})$ with $1\leq s_i\leq t_i,$ is adjacent only to $\langle p_k\rangle $.  Thus  $\langle p_k\rangle $ is pendent.\smallskip

  $(ii)$ Let $t_i\geq 2$ for $i\in [l]$ and $t_i=1$ for $i\in \{l+1, \dots,k\}$. Then any vertex of the form \\$\langle p_1^{s_1}p_2^{s_2}\dots p_l^{s_l}p_{l+1}\dots p_{l+i-1}p_{l+i+1}\dots p_{l+r}\rangle $ is adjacent to $\langle p_i\rangle $ only, where $1\leq i\leq r$ and $l+r=k$. Therefore, the number of pendant vertex of $\mathcal{G}(\mathbb{Z}_n)$ is $(k-l)\prod_{i=1}^lt_i$. 
\end{proof}

We conclude this work by classifying the finite cyclic group  $\mathbb{Z}_n$ such that the graph $\mathcal{G}(\mathbb{Z}_n)$ is Hamiltonian.

\begin{theorem}
    Let $n=\prod_{i=1}^{k}p_i^{t_i},$ where \(p_1,\ldots,p_k\) are distinct primes and \(t_i\geq1\). Then the graph $\mathcal{G}(\mathbb{Z}_n)$ is Hamiltonian if and only if $n=p_1^{r}p_2^{r}$, where $r\geq 2$.
\end{theorem}

\begin{proof}
Let $k=2$ with $t_1=t_2=r\geq 2$. Then
		\begin{equation*}
			\langle p_1\rangle \sim \langle {p_2}^{r}\rangle\sim \langle {p_1}^{r}\rangle\sim \langle {p_2}^{r-1}\rangle\sim \langle {p_1}^{r-1}\rangle \sim \cdots \sim \langle{p_2}\rangle\sim \langle{p_1}\rangle 
		\end{equation*}
		is a Hamiltonian cycle in $\mathcal{G}(\mathbb{Z}_n)$. Thus $\mathcal{G}(\mathbb{Z}_{{p_1^r}{p_2^r}})$ is a Hamiltonian graph.
		
		Conversely, suppose that $\mathcal{G}(\mathbb{Z}_n)$ is a Hamiltonian graph. We show that none of the following can occur: \(k\neq 2\), \(t_1=t_2=1\), or \(t_1\neq t_2\), since in each of these cases $\mathcal{G}(\mathbb{Z}_n)$ is not Hamiltonian. 
		
		Let $k=2$. If $t_1=t_2=1$, then $\mathcal{G}(\mathbb{Z}_n)\cong K_2$, which is not a Hamiltonian graph. Let $t_1\not=t_2$. Without loss of generality, assume that $t_1> t_2$. If $t_1-t_2=1$, then $\mathcal{G}(\mathbb{Z}_n)$ is a bipartite graph of odd order and hence is not a Hamiltonian Graph. If $t_1-t_2\geq 2$, then one can observe that the circumference of the complete bipartite graph $\mathcal{G}({\mathbb{Z}_n})$ is $2t_2<t_1+t_2=\vert \mathcal{G}(\mathbb{Z}_n)\vert$ and hence the graph $\mathcal{G}(\mathbb{Z}_n)$  is again not Hamiltonian.
		
		Finally, let $k\geq 3$. If $t_i=1$ for some $i\in [k]$, then by Corollary \ref{cormm}, $\mathcal{G}(\mathbb{Z}_n)$ contains a pendant vertex and hence is not a Hamiltonian graph. Now suppose that $t_i\geq 2$ for all $i\in [k]$, and let $t_k=\min\{t_i:i\in [k]\}$. Consider the sets $ A=\left\{\left\langle\prod_{i=1}^{k-1}p_i^{s_i}\right\rangle:1\leq s_i\leq t_i\right\}$ and $B=\left\{\langle p_k^{s_k}\rangle:1\leq s_k\leq t_k\right\}$. Every vertex in \(A\) is adjacent to every vertex in \(B\), while no two vertices in \(A\) are adjacent and no two vertices in \(B\) are adjacent. Thus, the subgraph of \(\mathcal{G}(\mathbb{Z}_n)\) induced by \(A\cup B\) is complete bipartite.
		Suppose that \(\mathcal{G}(\mathbb{Z}n)\) contains a Hamiltonian cycle. Then this cycle must contain a path whose internal vertices include all the vertices of \(A\) and whose two end vertices lie in \(B\). However,  Since $k\geq 3$ and $t_i\geq 2$ for all $i\in [k]$, we have $\vert A\vert=\prod_{i=1}^{k-1}t_i>t_k=\vert B\vert$. By the same argument as above, no such path can exist in the subgraph induced by $V(A\cup B)$, a contradiction. This concludes the proof.
\end{proof}

\section*{\textbf{Acknowledgement}}
 P.N. Choudhury\ was partially supported by
ANRF Prime Minister Early Career Research Grant ANRF/ECRG/2024/002674/PMS (ANRF, Govt.~of India), INSPIRE Faculty Fellowship Research Grant DST/INSPIRE/04/2021/002620 (DST, Govt.~of India).  B. Khanra was supported by the IIT Gandhinagar GAPD Fellowship SPON/70121 (Indian Institute of Technology Gandhinagar, India).

\end{document}